\documentclass [12pt,a4paper, reqno]{amsart}
\usepackage{amsmath,amssymb}
\usepackage{amstext, amsthm, ascmac,cases, mathtools}
\usepackage{amscd}
\usepackage{framed}
\usepackage{fancyhdr}
\usepackage[final]{pdfpages}
\usepackage{graphicx,xcolor}
\usepackage{tikz}
\usetikzlibrary{positioning,shapes,arrows}
\usetikzlibrary{3d,calc}
\tikzstyle{na} = [baseline=-.5ex]
\usetikzlibrary{intersections,math,patterns}
\usepackage{pgfplots}
\pgfplotsset{compat=1.6}

\pgfplotsset{soldot/.style={color=blue,only marks,mark=*}}
\pgfplotsset{holdot/.style={color=blue,fill=white,only marks,mark=*}}

\definecolor{cobalt}{rgb}{0.0, 0.28, 0.67}
\definecolor{brightcerulean}{rgb}{0.11, 0.67, 0.84}
 \usepackage[colorlinks=true,linkcolor=cobalt, setpagesize=false, linkbordercolor=cobalt]{hyperref}
\hypersetup{urlcolor=brightcerulean, citecolor=brightcerulean}
\usepackage{amsgen}

\makeatletter
\DeclareRobustCommand*{\bfseries}{%
  \not@math@alphabet\bfseries\mathbf
  \fontseries\bfdefault\selectfont
  \boldmath
}

\usepackage{mathrsfs}
\usepackage{dutchcal}
\usepackage{geometry}
\usepackage{wrapfig}

\usepackage{enumerate}

\DeclareMathOperator*{\supp}{supp}

\def\Xint#1{\mathchoice
{\XXint\displaystyle\textstyle{#1}}%
{\XXint\textstyle\scriptstyle{#1}}%
{\XXint\scriptstyle\scriptscriptstyle{#1}}%
{\XXint\scriptscriptstyle\scriptscriptstyle{#1}}%
\!\int}
\def\XXint#1#2#3{{\setbox0=\hbox{$#1{#2#3}{\int}$ }
\vcenter{\hbox{$#2#3$ }}\kern-.59\wd0}}
\def\dashint{\Xint-}

\def\Yint#1{\mathchoice
   {\YYint\displaystyle\textstyle{#1}}%
   {\YYint\textstyle\scriptstyle{#1}}%
   {\YYint\scriptstyle\scriptscriptstyle{#1}}%
   {\YYint\scriptscriptstyle\scriptscriptstyle{#1}}%
    \!\iint}
\def\YYint#1#2#3{{\setbox0=\hbox{$#1{#2#3}{\iint}$}
    \vcenter{\hbox{$#2#3$}}\kern-.51\wd0}}
\def\longdash{{-}\mkern-3.5mu{-}} 

\def\fiint{\Yint\longdash}

\makeatletter
  \def\thefootnote{\ifnum\c@footnote>\z@\leavevmode\lower.5ex%
      \hbox{$^{\@arabic\c@footnote)}$}\fi}
  \makeatother

\newtheoremstyle{mystyle}%                % Name
  {}%                                     % Space above
  {}%                                     % Space below
  {\itshape}%                                     % Body font
  {}%                                     % Indent amount
  {\bfseries }%                            % Theorem head font
  { }%                                    % Punctuation after theorem head
  { }%                                    % Space after theorem head, ' ', or \newline
  {\thmname{#1}\thmnumber{ #2}\thmnote{ (#3)}}%                                     % Theorem head spec (can be left empty, meaning `normal')

\theoremstyle{mystyle}

\usepackage{lipsum}

\makeatletter
\renewcommand{\@secnumfont}{\bfseries}
\renewcommand\section{\@startsection{section}{2}%
  \z@{-.5\linespacing\@plus-.7\linespacing}{.5\linespacing}%
  {\large\bfseries}}
\renewcommand\subsection{\@startsection{subsection}{3}%
  \z@{.5\linespacing\@plus.7\linespacing}{-.5em}%
  {\normalfont\bfseries}}
 \makeatother

\usepackage[title, toc, titletoc]{appendix}
\allowdisplaybreaks[4]

\newtheorem{thm}{Theorem}[section]
\newtheorem{prop}[thm]{Proposition}
\newtheorem{lem}[thm]{Lemma}
\newtheorem{dfn}{Definition}

\makeatletter

\@addtoreset{equation}{section}
\makeatother

\renewcommand{\to}{\longrightarrow}
\renewcommand{\rho}{\varrho}

  \renewcommand{\geq}{\geqq}
  \renewcommand{\leq}{\leqq}
   \renewcommand{\rho}{\varrho}

\def \bR{\mathbb R}

\def\dxy{\,{\rm d}x{\rm d}y}
\def\dxt{\,{\rm d}x{\rm d}t}
\def\dxyt{\,{\rm d}x{\rm d}y{\rm d}t}
\def\dx{\,{\rm d}x}

\def\ds{\,{\rm d}s}
\def\dt{\,{\rm d}t}
\def\dy{\,{\rm d}y}

\def \d{\,{\rm d}}

\def\supp{\,{\rm supp}}

\allowdisplaybreaks
\begin{document}
%%%%%%%%%%%%%%%%%%%%%%%%%%%%%%%%%%%%%%%%%%%%%%%%%%%%%%%%%%%%%%%%%%%%%%
%%%%%%%%%%%%%%%%%%%%%%%%%%%%%%%%%%%%%%%%%%%%%%%%%%%%%%%%%%%%%%%%%%%%%%
\titlepage

\title[Harnack's estimate for a mixed Local-Nonlocal Doubly Nonlinear Equation]{Harnack's estimate for a mixed local-nonlocal doubly nonlinear parabolic equation}

%\date{\textcolor{cyan}{\textbf{20210825 \,\,ver.}}}
%%%%%%%%%%%%%%%%%%%%%%%%%%%%%%%%%%%%%%%%%%%%
%%%%%%%%%%%%%%%%%%%%%%%%%%%%%%%%%%%%%%%%%%%%%%%%%%%%%%%%%%%%%%%%%%%%%%%%%%%%%%%%%%%

\author{Kenta Nakamura}
\address[Kenta Nakamura]{Headquarters for Admissions and Education, Kumamoto University, Kumamoto 860-8555, Japan}
\email{kntkuma21@kumamoto-u.ac.jp}

\date{\textcolor{blue}{\texttt{May, 2022}}}

 %%%%%%%pagenumber%%%%%%%%%%
\thispagestyle{firstpage}

\setcounter{page}{1}

%%%%%%%%%%%%%%%%%%%%%%
 
 %
\keywords{Local-Nonlocal doubly nonlinear equation, Harnack inequality, Caccioppoli type estimate, Moser's iteration}
\medskip
\subjclass[2010]{Primary: 35B45, 35M10, \quad Secondary: 35B09}

 \maketitle
 \begin{abstract}
We establish Harnack's estimates for positive weak solutions to a mixed local and nonlocal doubly nonlinear parabolic equation, of the type 
\[
\partial_t\left(|u|^{p-2}u\right)-\mathrm{div}\,\mathbf{A}\left(x,t,u,Du\right)+\mathcal{L}u(x,t)=0,
\]
where the vector field $\mathbf{A}$ satisfies the $p$-ellipticity and growth conditions and the integro-differential operator $\mathcal{L}$ whose model is the fractional $p$-Laplacian. 
All results presented in this paper are provided together with quantitative estimates.
 \end{abstract}

\tableofcontents

\section{Introduction and Main results}
In this paper, we consider the following mixed local-nonlocal doubly nonlinear equation of the form
\begin{equation}\label{maineq}
\partial_t\left(|u|^{p-2}u\right)-\mathrm{div}\,\mathbf{A}\left(x,t,u,Du\right)+\mathcal{L}u(x,t)=0\quad \mathrm{in}\,\,\Omega_T:=\Omega \times (0,T),
\end{equation}
where $p>1$ denotes the summability  and the operator $\mathbf{A}: \Omega_T \times \bR \times \bR^n \to \bR^n$ is a Carath\'{e}odory function fulfill  the $p$-ellipticity and growth conditions:
\[
\begin{cases}
\mathbf{A}(x,t,u,\xi) \cdot \xi \geq c_0 |\xi|^p, \\
\left|\mathbf{A}(x,t,u,\xi)\right|\leq c_1|\xi|^{p-1}
\end{cases}
\]
for almost every $(x,t) \in \Omega_T$ and every $(u, \xi) \in \bR\times \bR^n$ with  positive constants $c_0$ and $c_1$, where the integro-differential operator $\mathcal{L}$ is defined by
\begin{align}\label{L}
\mathcal{L}u(x,t)&:=\,\mathrm{P.V.}\int_{\mathbb{R}^n} K(x,y,t)|u(x,t)-u(y,t)|^{p-2}(u(x,t)-u(y,t))\dy\notag \\[3mm]
&=\lim_{\varepsilon \searrow 0}\int_{\mathbb{R}^n \setminus B_\varepsilon(x)} K(x,y,t)|u(x,t)-u(y,t)|^{p-2}(u(x,t)-u(y,t))\dy,
\end{align}
where the symbol $\mathrm{P.V.}$ means ``in the principal value sense''. Further, the kernel $K(x,y,t)$ is assumed to be a nonnegative measurable function on $\bR^n \times \bR^n \times (0,T)$ fulfilling
\[
K(x,y,t)=K(y,x,t)
\]
and
\[
\dfrac{\Lambda^{-1}}{|x-y|^{n+sp}} \leq K(x,y,t) \leq \dfrac{\Lambda}{|x-y|^{n+sp}}
\]
for almost every $(x,y,t) \in \bR^n \times \bR^n \times (0,T)$, with  $\Lambda \geq 1$ being a constant and $s \in (0,1)$ being fractional differentiability.
\medskip

Our equation~\eqref{maineq} contains the prototype equation
\begin{equation}\label{maineq2}
\partial_t\left(|u|^{p-2}u\right)-\triangle_pu+(-\triangle)_p^su=0,
\end{equation}
where $\triangle_p u :=\mathrm{div}\left(|Du|^{p-2}Du\right)$ denotes the classical $p$-Laplacian with  $Du$ being the derivative with respect to the space-variable $x$. We know that $\triangle_p u$ is the gradient vector field of the $p$-energy $\displaystyle W^{1,p}_0(\Omega) \ni u \mapsto \int_\Omega |Du|^p\dx$. Analogously, the fractional $p$-Laplace operator $(-\triangle)_p^s u$, defined by~\eqref{L} with $\Lambda=1$, can be understood as the gradient vector field $\nabla \mathcal{E}(u)$ on $W^{s,p}_0(\Omega)$, where $\mathcal{E}(u)$ denotes an energy functional on $W^{s,p}_0(\Omega)$, defined by
\[
W^{s,p}_0(\Omega) \ni u \quad \mapsto \quad \mathcal{E}(u):=\frac{1}{p}\iint_{\bR^n \times \bR^n} \frac{|u(x)-u(y)|^p}{|x-y|^{n+sp}}\dxy.
\]
Here the fractional Sobolev space $W^{s,p}_0(\Omega)$ (see e.g.~\cite{Hitchhiker} for details) is given by
\begin{equation*}
W_0^{s,p}(\Omega):=\Big\{ \|u\|_{W^{s,p}(\bR^n)}<\infty \,:\,u=0 \,\,\,\textrm{a.e.\,\,in}\,\,\mathbb{R}^n \setminus \Omega \Big\},
\end{equation*}
where the norm $\|\cdot \|_{W^{s,p}(\bR^n)}$ is defined by
\[
\|u\|_{W^{s,p}(\bR^n)}:=\left(\int_{\bR^n}|u(x)|^p\dx\right)^{\frac{1}{p}}+\left(\iint_{\bR^n \times \bR^n}\dfrac{|u(x)-u(y)|^p}{|x-y|^{n+sp}}\dxy \right)^{\frac{1}{p}}.
\]
It is worth remarking that, for a Sobolev function $u \in W^{1,p}_{\textrm{loc}}(\Omega)$, there holds that
\[
\lim_{s\nearrow1}\,(1-s)\iint_{B_\rho \times B_\rho}\frac{|u(x)-u(y)|^p}{|x-y|^{n+sp}}\dxy=C(n,p)\int_{B_\rho}|Du(x)|^p\dx \qquad \forall B_\rho \Subset \Omega,
\]
which can be seen in~\cite{BBM}. Furthermore, as expected, the solution of $(-\triangle)_p^su=0$ converges strongly to $-\triangle_pu=0$ in $L^p(\Omega) \cap W^{1,p}_{\textrm{loc}}(\Omega)$ as $s\nearrow 1$. We refer to~\cite[Sect.1.4, Comments]{BL} for a precise description of the limit case $s\nearrow 1$.
The fundamental tools and regularity theory of the fractional Sobolev space and interesting nonlocal problems during the last decades are addressed in the comprehensive literatures~\cite{Hitchhiker, KP} and references therein.
\medskip

%The regularity theory for the doubly nonlinear parabolic problems has been attention and
%studied in last decades. In fact, there are   but there are 
The prototype equation~\eqref{maineq2} is considered as the mixed type of the so-called~\emph{Trudinger equation}
\begin{equation}\label{Trudinger}
\partial_t\left(|u|^{p-2}u\right)-\triangle_p u=0
\end{equation}
and the~\emph{nonlocal Trudinger equation}
\begin{equation}\label{Frac. Trudinger}
\partial_t\left(|u|^{p-2}u\right)+(-\triangle)_p^s u=0.
\end{equation}
One of the features of Eq.s~\eqref{Trudinger} and~\eqref{Frac. Trudinger} is the homogeneity, that is, the solution to~\eqref{Trudinger} or~\eqref{Frac. Trudinger} can be scaled by any scale factor. As seen in the literatures~\cite{BDKS1, BDKS2, BDL, Kuusi-Misawa-Nakamura1, Kuusi-Misawa-Nakamura2}, this homogeneity is advantage in De Giorgi's measure theoretic lemmata.
In 1968, by the use of Moser's iteration, Trudinger first proved a Harnack inequality for nonnegative solutions to~\eqref{Trudinger} in his pioneer work~\cite{Trudinger}. Later, in the case $p>2$, Gianazza and Vespri~\cite{GV} extended the Trudinger's result to more general equation, replaced $\triangle_p u$ by $\mathbf{A}=\mathbf{A}(x,t,u,Du)$, where the vector field $\mathbf{A}$ satisfies the $p$-ellipticity and growth conditions. In a doubling Borel measure framework, Kinnunen and Kuusi~\cite{Kinnunen-Kuusi} succeeded proving a Harnack inequality for positive weak solutions to~\eqref{Trudinger}. Their proofs are based on the Moser scheme combined with the Sobolev and Caccioppoli type inequalities. 
\medskip  
 
For the nonlocal parabolic problem, there are many literatures~\cite{AABP, BLS, DZZ, KP, KS, Mazon-Rossi-Toledo, Puhst, Vazquez} and references therein about a nonlocal parabolic equation of the form
\begin{equation}\label{nonlocal eq.}
\partial_t u(x,t)+\mathcal{L}u(x,t)=0,
\end{equation}
although we will give the brief overview of literatures related to the regularity results for~\eqref{nonlocal eq.}. %The existence and uniqueness of strong solutions and its asymptotic behavior are established by V\'{a}zquez~\cite{Vazquez}. 
Kassmann and Schwab~\cite{KS} showed a weak Harnack inequality and the H\"{o}lder regularity for solutions of the nonlocal heat equation~\eqref{nonlocal eq.} with $p=2$ and extra force term $f \in L^\infty$. Alternatively, Kim~\cite{Kim} showed a Harnack inequality with nonlocal tail for~\eqref{nonlocal eq.} with $p=2$. Later, Str\"{o}mqvist ~\cite{Stromqvist1, Stromqvist2} obtained the local boundedness and a Harnack inequality with nonlocal tail for weak solutions to~\eqref{nonlocal eq.} in the case $p>2$. As far as we know, this is a first contribution of Di Castro, Kuusi and Palatucci~\cite{DiCKP1, DiCKP2} in the parabolic setting. Alternatively, in the case $p>2$, Ding, Zhang and Zhon~\cite{DZZ} proved the local boundedness and the H\"{o}lder regularity for weak solutions to~\eqref{nonlocal eq.} with a source term $f=f(x,t,u)$ satisfying some structural conditions. Very recently, Brasco, Lindgren and Str\"{o}mqvist~\cite{BLS} showed the H\"{o}lder regularity for (local) weak solutions in the case $p \geq 2$, whose proof is completely different to previous approaches in view of using the iterated discrete differentiation method together with a Morrey type embedding. On the contrast, in the doubly nonlinear setting, Banerjee, Garain and Kinnunen~\cite{BGK} first proved with the local boundedness of positive solutions to~\eqref{Frac. Trudinger}. 
The regular form of~\eqref{maineq}, that is, 
\begin{equation}\label{mixed}
\partial_tu-\triangle_p u+(-\triangle)^s_pu=0,
\end{equation}
where the fractional order $s \in (0,1)$ and summability $p>1$, is motivated by not a just purely mathematical interest, but also biological modeling. Indeed, there is an application of the mixed local and nonlocal operator $u \mapsto -\triangle_pu+(-\triangle)_p^s u$ to a logistic equation from the viewpoint of biological problems, see~\cite{DPLV} and references therein. Very recently, Garain and Kinnunen ~\cite{GK2, GK3} showed a Harnack estimate and the local H\"{o}lder regularity for weak solutions to~\eqref{mixed} in the case $p>1$. Alternatively, Fang, Shang and Zhang~\cite{FSZ} showed the local boundedness and the H\"{o}lder regularity for weak solutions to~\eqref{mixed}.
\medskip

To the best of our knowledge, this paper contributes to new results for positive weak solutions to our mixed local-nonlocal doubly nonlinear parabolic equation~\eqref{maineq} and the technical novelties of this paper are the local and nonlocal mollification arguments with a detailed description. Although, as long as we employ the approach in this paper, it is worth remarking that, the positivity condition of solutions cannot be removable readily, because the power nonlinearity with respect to the possibly sign-changing solution itself makes the situation more difficult. Therefore we need to employ another approach like the so-called ``expansion of positivity''. This phenomenon often occurs in the usual doubly nonlinear parabolic equation of the form
\[
\partial_t\left(|u|^{q-1}u\right)-\triangle_p u=0,\quad q>0,\quad p>1.
\]
In the doubly nonlinear framework, we make the full use of the De Giorgi's measure theoretic approach because it is very flexible. We refer to~\cite{BDKS1, BDKS2, BDL, Kuusi-Misawa-Nakamura1, Kuusi-Misawa-Nakamura2} for a detailed description.
\medskip

Before formulating the main results, we need to introduce the notion of weak solutions to~\eqref{mainthm} as follows.
Let $\mathcal{T}$ be the class of test functions defined by
\[
\mathcal{T} : = \left\{ \varphi \in L^\infty(0,T\,; L^p(\Omega)) \cap L^p (0, T\,; W^{1, p}_0 (\Omega)) \,\,\Bigg|
\left.\begin{array}{c}\partial_t \varphi \in L^p(\Omega_T), \\[1mm]\varphi (x,0)= \varphi (x,T)=0 \quad \textrm{a.e.}\,\, x \in \Omega \end{array}\right.
\right\},
\]
where the Sobolev space $W_0^{1,p}(\Omega)$ with zero boundary value is defined by
\[
W^{1,p}_0(\Omega):=\left\{u \in W^{1,p}(\Omega): u=0\,\,\textrm{in}\,\,\bR ^n\setminus \Omega\right\}.
\]
\begin{dfn}[Weak solution]\label{weak sol}\normalfont
Suppose that the vector field $\mathbf{A}$ and the nonlocal operator $\mathcal{L}$ satisfy the conditions~\eqref{S1}--\eqref{S3}. A measurable function $u=u(x,t)$ defined on $\bR^n_T:=\bR^n \times (0,T)$  in the class 
\[u \in L^\infty(0,T\,; L^p(\Omega)) \cap L^p(0,T\,; W^{1,p}(\Omega)) \cap L^\infty(0,T\,; L^{p-1}_{sp}(\bR^n))
\]
is a weak sub(super)-solution to~\eqref{maineq} iff
\begin{align}\label{D2}
&\iint_{\Omega_T}\left(-|u|^{p-2}u\cdot \partial_t\phi+\mathbf{A}\left(x,t,u,Du\right)\cdot D\phi\right)\dxt \notag\\[2mm]
&+\int_0^T\iint_{\bR^n \times \bR^n}U(x,y,t)K(x,y,t)\left(\phi(x,t)-\phi(y,t)\right)\dxyt 
\!\!\!\!\!\!\!\!\!\!\!\!\!\!\!\!\!\!\!\!\!\!\!\!\!\!\!\!
\begin{split}
&\,\leq\\[-2mm]
&(\geq)
\end{split}
\,\,0
\end{align}
holds for every nonnegative testing function $\phi \in \mathcal{T}$ with the shorthand notation
\[
U(x,y,t):=|u(x,t)-u(y,t)|^{p-2}(u(x,t)-u(y,t)).
\]
We say that $u$ is a weak solution to~\eqref{maineq} if and only if $u$ is simultaneously a weak super and subsolution to~\eqref{maineq}.
\end{dfn}
We remark that, following the argument as in~\cite{BDV, Sturm}, the time continuity in $L^p$ for weak solutions $u$ to~\eqref{maineq} in the sense of Definition~\ref{weak sol} can be derived, that is,
\[
u \in C([0,T]\,;L^p(\Omega)).
\]
The precise proof is seen in~\cite[Proposition 3.4]{Nakamura}.
\medskip

We are ready to take on the main theorem as follows.
\begin{thm}[Weak Harnack estimate]\label{mainthm}
Let $p >1 $, $s \in (0,1)$ and fix the power $q$ with $0<q<\frac{n+p}{n}(p-1)$. Suppose that the vector field $\mathbf{A}$ and the integro-differential operator $\mathcal{L}$ satisfy the conditions~\eqref{S1}--\eqref{S3}. Suppose further that a weak supersolution $u$  to~\eqref{maineq} in the sense of Definition~\ref{weak sol} satisfies $u \geq m>0$ in $\bR^n \times (0,T)$. Then, for any $\delta \in (0,1)$ there exists a constant $C\equiv C(n,s,p,c_0,c_1,\Lambda, \delta,q)$ such that
\[
\left(\fiint_{Q_{\delta\rho}^-(z_0)}u^q\dxt\right)^{\frac{1}{q}} \leq C \inf_{Q^+_{\delta \rho}(z_0)}u
\]
holds whenever concentric space-time cylinders $Q_{\delta\rho}^\pm (z_0) \subset Q_{\rho} (z_0) \Subset \Omega_T$ with $0<\rho \leq 1$.
\end{thm}
\medskip

As a by-product of Theorem~\ref{mainthm}, we have the following theorem.
\begin{thm}[Harnack estimate]\label{mainthm2}
Suppose that the vector field $\mathbf{A}$ and the integro-differential operator $\mathcal{L}$ satisfy the conditions~\eqref{S1}--\eqref{S3}. With $p>1$, $s \in (0,1)$  let $u$ be a weak solution to~\eqref{maineq} in the sense of Definition~\ref{weak sol} fulfilling $u \geq m>0$ in $\bR^n \times (0,T)$. Then, for any $\sigma \in (0,1)$ there exists a constant $C\equiv C(n,s,p,c_0,c_1,\Lambda,\sigma)$ such that
\[
\sup_{Q^-_{\sigma \rho}(z_0)}u \leq C\inf_{Q^+_{\sigma \rho}(z_0)}u
\]
holds whenever concentric space-time cylinders $Q_{\sigma\rho}^\pm (z_0) \subset Q_{\rho} (z_0) \Subset \Omega_T$ with $0<\rho \leq 1$.
\end{thm}

\medskip

\noindent
\textbf{Structure of the paper}
\medskip

In Sect.~\ref{Sect. 2}, we list the notation that will be used throughout the paper, then state the structure assumption and finally collect some auxiliary tools. In Sect.~\ref{Sect. 3}, we derive quantitative estimates for supersolutions; in particular, we derive the Reversed H\"{o}lder inequality (Proposition~\ref{reverse Holder})  in Sect. ~\ref{Sect. 3.1}, then prove the log-type estimates (Lemma~\ref{log-type Caccioppoli} and Proposition~\ref{log-type est.}) and a useful lemma (Lemma~\ref{supsub}) in Sect.~\ref{Sect. 3.2}. Sect.~\ref{Sect. 4} is devoted to the local boundedness  (Proposition~\ref{local bounds}) for subsolutions. Sect.s~\ref{Sect. 5} and~\ref{Sect. 6} give the proofs of Theorems~\ref{mainthm} and~\ref{mainthm2}, respectively. In Appendix~\ref{Appendix A} we give the full proof of a Caccioppoli type estimate for supersolutions (Lemma~\ref{Caccioppoli1}). Finally, in Appendix~\ref{Appendix B}, we prove Lemmata~\ref{log-type Caccioppoli lemma} and~\ref{supsub lemma}.
\medskip

\noindent
\textbf{Acknowledgements}
\medskip

Kenta Nakamura acknowledges the  support by Grant-in-Aid for Young Scientists Grant \#21K13824 (2021) at Japan Society for the Promotion of Science. 
%%%%%%%%%%%%%%%%%%%%%%%%%%%%%%%%
%%%%%%%%%%%%%%%%%%%%%%%%%%%%%%%%
\section{Notation and preliminary materials}\label{Sect. 2}

In this section, we will record our notation, describe the structural assumptions on~\eqref{maineq} and collect some auxiliary materials that will be helpful at various stages of the paper.

%%%%%%%%%%%%%%%%%%%%%%%%%%%%%%%%
%%%%%%%%%%%%%%%%%%%%%%%%%%%%%%%%
%%%%%%%%%%%%%%%%%%%%%%%%%%%%%%%%
\subsection{Notation}
In this brief section we introduce the notation that will be used.
In the present paper,  we shall fix exponents $p >1$ and $0<s<1$, while a bounded open subset $\Omega$ of $\bR^n$, with $n \geq 2$. For $T \in (0,\infty)$, let $\Omega_T:=\Omega \times (0,T)$ be a space-time cylinder. We shall denote in a standard way $B_\rho(x_0):=\left\{ x \in \bR^n: |x-x_0|<\rho\right\}$, the open ball with center $x_0 \in \bR^n$ and radius $\rho>0$. We omit denoting $B_\rho$ instead of $B_\rho(x_0)$ when being center from the context, all the balls considered will share the same center. For $z_0=(x_0,t_0) \in \bR^n \times \bR$, let us define space-time cylinders as follows
\begin{align*}
Q^{-}_{\rho}(z_0)&:=B_\rho(x_0)\times (t_0-\rho^p,t_0),\\[2mm]
Q^{+}_{\rho}(z_0)&:=B_\rho(x_0)\times (t_0,t_0+\rho^p), \\[2mm]
Q_{\rho}(z_0)&:=B_\rho(x_0)\times (t_0-\rho^p,t_0+\rho^p),
%&Q^+_{\rho}(z_0):=B_\rho(x_0)\times (t_0, t_0+\rho^p), 
\end{align*}
where $\rho>0$ is a radius and $\rho^p$ is a time length. The $\lambda$-dilate of $Q^{\pm}_\rho(z_0)$ and $Q_{\rho}(z_0)$ with $\lambda>0$ are denoted by $\lambda Q^{\pm}_\rho(z_0):=Q^{\pm}_{\lambda \rho}(z_0)$ and $\lambda Q_\rho(z_0):=Q_{\lambda \rho}(z_0)$, respectively. 
\medskip

As customary, we write $w(t):=w(\cdot, t)$, which means the time-slice value at time $t \in (0,T)$ for functions $w$, defined on a space-time region.
\medskip

With $\mathcal{B} \subset \bR^k,\,k \geq 1$ being a measurable subset with finite and positive measure $|\mathcal{B}|>0$, and with $g:\mathcal{B} \to \bR$, being an integral function, we denote by
\[
(g)_{\mathcal{B}}\equiv \dashint_{\mathcal{B}}g(x)\dx:=\frac{1}{|\mathcal{B}|}\int_{\mathcal{B}}g(x)\dx
\]
its integral average. In this paper, all the measures addressed will be the Lebesgue measure on Euclidean space $\bR^k$, with $k \geq 1$.
\medskip

For open sets $E$ and $F$ in $\bR^k$,\,$k \geq 1$,  the symbol $E \Subset F$ denotes the closure of $E$ is compactly contained in $F$.

%For $u,k \in \bR$, we write
%
%\[
%(u-k)_+:=\max\{u-k\,,\,0\},\quad (u-k)_-:=\max\{-(u-k)\,,\,0\}.
%\]
%
%Furthermore, we abbreviate 
For a set $S \subset \bR^{n} \times \bR$ we write
\begin{align*}
S \cap \{u  >k\}:=\left\{(x,t) \in S: u(x,t)>k\right\}, \\[2mm]
S \cap \{u  <k\}:=\left\{(x,t) \in S: u(x,t)<k\right\}.
\end{align*}
Next, we briefly recall the nonlocal tail that naturally appears when dealing with nonlocal operators like $\mathcal{L}$. This nonlocal  quantity is originally introduced in~\cite{DiCKP1}.

Let $p \in [1,\infty)$ and $s \in (0,1)$. We quantify the \emph{nonlocal tail} in $B_\rho(x_0)$ as follows:
\[
\mathrm{Tail}(v,B_\rho(x_0)):=\left(\rho^{sp}\int_{\bR^n \setminus B_\rho(x_0)}\dfrac{|v(x)|^{p-1}}{|x-x_0|^{n+sp}}\dx\right)^{\frac{1}{p-1}}.
\]
The \emph{tail space} $L^{p-1}_{sp}(\bR^n)$ is defined by requiring that $v \in L^{p-1}_{sp}(\bR^n)$ if and only if $v \in L^{p-1}_{\mathrm{loc}}(\bR^n)$ and 
\[
\mathrm{Tail}(v,B_\rho(x_0))<\infty\qquad \forall x_0 \in \bR^n, \,\,\forall \rho>0.
\]
By definition, we see that
\[
L^{p-1}_{sp}(\bR^n):=\left\{v \in L^{p-1}_{\mathrm{loc}}(\bR^n): \int_{\bR^n}\frac{|v(x)|^{p-1}}{(1+|x|)^{n+sp}}\dx<\infty \right\}.
\]
The \emph{parabolic nonlocal tail} of a function $v \in L^\infty(t_0-\rho^p, t_0+\rho^p\,;\,L^{p-1}_{\mathrm{loc}}(\bR^n))$  is defined by
\[
\mathrm{Tail}_\infty(v, Q_\rho(z_0)):=\left(\sup_{t_0-\rho^p<t<t_0+\rho^p}\rho^{sp}\int_{\bR^n \setminus B_\rho(x_0)}\dfrac{|v(x,t)|^{p-1}}{|x-x_0|^{n+sp}}\dx\right)^{\frac{1}{p-1}},
\]
which is presented in~\cite{Stromqvist2}. This quantity naturally appears in Theorem~\ref{local bounds with tail}, as seen below.
\medskip

Finally, we will list the general notation. Throughout the paper, $c$, $C$, $\cdots $ denote different positive constants in a given context. Relevant dependencies on parameters will be emphasized using parentheses, e.g., $c\equiv c(n,s,p)$ means that $c$ depends on $n, s$ and $p$. For the sake of readability, the dependencies of the constants will be often omitted within the chains of estimates.
Furthermore, the equation number $(\,\cdot\,)_\ell$ denotes the $\ell$-th line of the Eq. $(\,\cdot\,)$.
\subsection{Setting}
In order to consider weak solutions to the mixed local and nonlocal doubly nonlinear parabolic equation~\eqref{maineq}, we impose the assumption that will be used in the course of paper. The vector field $\mathbf{A}=\mathbf{A}(x,t,u,\xi):\Omega_T \times \bR \times \bR^n \to \bR^n$ appearing on~\eqref{maineq} is assumed to be measurable with respect to $(x,t) \in \Omega_T$ for every $(u,\xi)  \in \bR \times \bR^n$ and continuous with respect to $(u,\xi)$ for almost everywhere $(x,t) \in \Omega_T$. We further suppose that $\mathbf{A}$ fulfills the structure condition
\begin{equation}\label{S1}
\begin{cases}
\mathbf{A}(x,t,u,\xi) \cdot \xi \geq c_0 |\xi|^p, \\
\left|\mathbf{A}(x,t,u,\xi)\right|\leq c_1|\xi|^{p-1}
\end{cases}
\end{equation}
with $p>1$ and the structure constants $c_0$ and $c_1$. The integro-differential operator $\mathcal{L}$ appearing on~\eqref{maineq} is taken in the Cauchy principal value: 
\[
\mathcal{L}u(x,t):=\,\mathrm{P.V.}\int_{\mathbb{R}^n} K(x,y,t)|u(x,t)-u(y,t)|^{p-2}(u(x,t)-u(y,t))\dy,
\]
where the kernel $K: \bR^n \times \bR^n \times (0,T) \to [0,\infty)$ is a measurable function fulfilling the symmetric property
\begin{equation}\label{S2}
K(x,y,t)=K(y,x,t)
\end{equation}
and
\begin{equation}\label{S3}
\dfrac{\Lambda^{-1}}{|x-y|^{n+sp}} \leq K(x,y,t) \leq \dfrac{\Lambda}{|x-y|^{n+sp}}
\end{equation}
with $\Lambda \geq 1$ and a fractional order $s \in (0,1)$ for every $(x,y,t) \in \bR^n \times \bR^n \times (0,T)$. Notice that, when $\Lambda=1$ the operator $\mathcal{L}$ coincides the fractional $p$-Laplacian $(-\triangle)_p^s$.
\subsection{Auxiliary materials}\label{Sect. 2.2}
In this subsection, we collect the auxiliary material used throughout the paper. %

Firstly, we recall a very useful inequality that controls the fractional integral term, which is a variant version of~\cite[Lemma 3.1]{DiCKP1}.
\begin{lem}\label{Gamma}
Let $p \geq 1$ and $\varepsilon \in (0,1]$. Then, 
\[
|a|^p\leq |b|^p+c\varepsilon|b|^p+(1+c\varepsilon)\varepsilon^{1-p}|a-b|^p.
\]
holds true whenever $a,\,b \in \bR^k,\,\,k \geq 1$, where $c=c(p):=2^p(p-1)^p$.
\end{lem}
\begin{proof}
For the reader's convenience we give the short proof that is slightly different to~\cite[Lemma 3.1]{DiCKP1}, but similar to~\cite[Lemma 4.3]{Cozzi}. 

Due to  the fundamental of calculus and Young's inequality, we infer that, for any $\varepsilon \in (0,1]$
\begin{align}\label{Gamma eq.1}
\big(|b|+|a-b|\big)^p-|b|^p &=p\int_0^{|a-b|}\left(|b|+t\right)^{p-1}\dt \notag\\[3mm]
&\leq p\big(|b|+|a-b|\big)^{p-1}|a-b| \notag\\[3mm]
&\leq (p-1)\varepsilon \big(|b|+|a-b|\big)^{p}+\varepsilon^{1-p}|a-b|^p.
\end{align}
Now, taking $\varepsilon=\frac{1}{2(p-1)}$ yields in particular that
\[
\big(|b|+|a-b|\big)^{p} \leq 2|b|^p+2^p(p-1)^{p-1}|a-b|^p.
\]
Inserting this back to~\eqref{Gamma eq.1}, we gain
\begin{align*}
\big(|b|+|a-b|\big)^p-|b|^p &\leq 2(p-1)\varepsilon |b|^p+\left[\,2^p(p-1)^{p}\varepsilon+\varepsilon^{1-p}\,\right]|a-b|^p\\[2mm]
&\leq c(p)\varepsilon|b|^p+(1+c(p)\varepsilon)\varepsilon^{1-p}|a-b|^p
\end{align*}
with $c(p)=2^p(p-1)^{p}$. Thus, this in turn implies that
\[
|a|^p-|b|^p \leq \big(|b|+|a-b|\big)^p-|b|^p \leq c(p)\varepsilon|b|^p+(1+c(p)\varepsilon)\varepsilon^{1-p}|a-b|^p,
\]
as desired.
\end{proof}
We further retrieve the algebraic estimate; the proof is in~\cite[Lemma 2.2]{Acerbi-Fusco} in the case $0<\beta<1$ and in~\cite[inequality (2.4)]{Giaquinta-Modica} in the case $\beta>1$.
\begin{lem}[Algebraic inequality I]\label{Algs}
For every $\beta>0$ there exists a constant $c(\beta)$ such that
\[
c^{-1}\Big||\xi|^{\beta-1}\xi-|\eta|^{\beta-1}\eta\Big| \leq \left(|\xi|+|\eta|\right)^{\beta-1}|\xi-\eta| \leq c\Big||\xi|^{\beta-1}\xi-|\eta|^{\beta-1}\eta\Big|
\]
holds true whenever $\xi,\eta \in \bR$.
\end{lem}

%
%We also recall another algebraic inequality retrieved from~\cite[Lemma 8.3]{Giusti}.
The above algebraic inequality allows us to derive the following useful inequality:
\begin{lem}[Algebraic inequality II]\label{Algs'}
For all $\alpha \in (1,\infty)$ there is a constant $c\equiv c(\alpha)$ such that 
\[
c^{-1} \leq \frac{(|\xi|^{\alpha-2}\xi-|\eta|^{\alpha-2}\eta)(\xi-\eta)}{  (|\xi|+|\eta|)^{\alpha-2}(\xi-\eta)^2}\leq c
\]
holds whenever $\xi,\,\eta \in \bR$ with $\xi \neq \eta$; in particular,
\begin{equation}\label{algs'}
(|\xi|^{\alpha-2}\xi-|\eta|^{\alpha-2}\eta)(\xi-\eta) \geq 0\qquad \forall \xi,\,\eta \in \bR.
\end{equation}
\end{lem}

The next inequality is another necessary tool to control the fractional integral term, whose proof is in~\cite[Lemma 2.9]{BGK}.

\begin{lem}\label{fractional est.}
Let $a,b >0$, $\tau_1,\tau_2 \geq 0$. Then for all $p>1$ there exists a constant $c\equiv c(p)$ such that
\begin{align*}
|b-a|^{p-2}(b-a)\left(\tau_1^pa^{-\varepsilon}-\tau_2^pb^{-\varepsilon}\right)&\geq c\zeta(\xi)\Big|\tau_2b^{\frac{\alpha}{p}}-\tau_1a^{\frac{\alpha}{p}}\Big|^p \\[3mm]
&\quad \quad \quad \quad-\left(\zeta(\varepsilon)+1+\varepsilon^{-(p-1)}\right)|\tau_2-\tau_1|^p\left(b^\alpha+a^\alpha\right),
\end{align*}
where $\varepsilon \in (0,p-1)$, $\alpha:=p-1-\varepsilon$ and the function $\varepsilon \mapsto \zeta(\varepsilon)$ is explicitly given by
\[
\zeta(\varepsilon):=\begin{cases}
\dfrac{\varepsilon p^p}{\alpha}\quad &\textrm{if}\,\,\,0 <\alpha<1\\[4mm]
\varepsilon \left(\dfrac{p}{\alpha}\right)^p \quad& \textrm{otherwise}.
\end{cases}
\]
\end{lem}
\medskip

We next deduce the Gagliardo-Nirenberg inequality of parabolic type, which is retrieved from~\cite[Chapter I.4]{DiBenedetto1}. The proof is also seen in~\cite[Lemma 2.3]{Nakamura}.

\begin{lem}[Gagliardo-Nirenberg inequality]\label{GN}
Let  $1\leq p,r <\infty$ and $0\leq t_1<t_2 \leq T$. Assume that
\[
v \in L^\infty(t_1,\,t_2\,; L^r(B_\rho(x_0))) \cap L^p(t_1,\,t_2\,;W^{1,p}(B_\rho(x_0))). 
\]
Then there exists a constant $c\equiv c(n,p,r)$ such that
\begin{equation}\label{GNeq.}
\fiint_{Q_{\rho; t_1,t_2}}|v|^{p\frac{n+r}{n}}\dxt 
\leq c\rho^{p}\left(\sup_{t\in (t_1,t_2)}\dashint_{B_\rho(x_0)}|v(t)|^r\dxt\right)^{\frac{p}{n}}\fiint_{Q_{\rho; t_1,t_2}}\left(\,|Dv|^p+\left|\frac{v}{\rho}\right|^p \,\right)\dxt,
\end{equation}
where $Q_{\rho\,; \,t_1,t_2}:=B_\rho(x_0) \times (t_1,t_2)$.
\end{lem}
The following inequality holds on the Sobolev space $W_0^{1,p}(\Omega)$, whose proof can be seen in~\cite[Lemma 2.1]{BSM}.

\begin{lem}\label{Wsp ineq.}
Let $1 <p<\infty$, $0<s<1$ and $\Omega \subset \bR^n$ be a bounded domain with the Lipschitz boundary. Then there exists a constant $c\equiv c(n,s,p,\Omega)$ such that
\[
\iint_{\bR^n \times \bR^n}\dfrac{\left|v(x)-v(y)\right|^p}{|x-y|^{n+sp}}\dxy \leq c \int_\Omega |Dv|^p\dx
\]
holds wherever $v \in W^{1,p}_0(\Omega)$.
\end{lem}

%
%

%We also list the~\emph{fast geometric convergence lemma} which will be employed later. See~\cite[Chapter I.4, Lemma 4.1]{DiBenedetto1} for details.

%\begin{lem}[Fast geometric convergence lemma]\label{Fast geometric convergence}
%Let $(Y_i)_{i=0}^\infty$ be a sequence of positive numbers, satisfying the recursive inequalities
%
%\[
%Y_{i+1} \leq c\,b^iY_i^{1+\beta},\quad j=0,1,2,\ldots,
%\]
%
%where $c, \beta>0$ and $b>1$ are given constants independent of $i$. Assume that the initial datum $Y_0$ satisfies
%
%$
%Y_0 \leq c^{-1/\beta}b^{-1/\beta^2}.
%$
%
%Then, $Y_i \to 0$ as $i \to \infty$ holds true.
%\end{lem}
%
%

We finally present the so-called ``simple but fundamental lemma'' due to Giaquinta and Giusti~\cite[Lemma 1.1]{Giaquinta-Giusti} or~\cite[Lemma 6.1, Page 191]{Giusti}. 
\begin{lem}\label{iteration lemma}
Let  $A,\,B \geq 0$, $\beta>0$ and $0<\sigma<\tau$. Suppose that any nonnegative bounded function $\mathcal{Z}=\mathcal{Z}(s)$ defined on $[\sigma, \tau]$ satisfies
\[
\mathcal{Z}(s) \leq \frac{3}{4} \mathcal{Z}(t)+A(t-s)^{-\beta}+B
\]
for every $\sigma \leq  s<t \leq \tau$. Then there exists a constant $c$, depending only on $\beta$, such that
\[
\mathcal{Z}(\sigma) \leq c\left(A(\tau-\sigma)^{-\beta}+B\right).
\]

\end{lem}
\subsection{Mollification in time}
In this subsection we introduce the \emph{exponential mollification in time}, originally devised in~\cite{KL}. This mollification is a breakthrough tool that can overcome the lack of weak differentiable in time for solutions. For this reason, this technique is  applied to various doubly nonlinear equations, as seen in the literatures~\cite{BDV,BDKS1,BDKS2, BDL, Nakamura}. With $E \subset \bR^k$ being a bounded domain,  let us define, for $v \in L^1(E_T)$ and $h \in (0,T)$, 
\begin{equation}\label{def. mollification}
[v]_h(x,t):=\dfrac{1}{h}\int_0^te^{\frac{s-t}{h}}v(x,s)\ds,\quad (x,t) \in E \times [0,T].
\end{equation}
The backward version of $[v]_h$ is given by
\[
[v]_{\bar{h}}(x,t):=\dfrac{1}{h}\int_t^Te^{\frac{t-s}{h}}v(x,s)\ds,\quad (x,t) \in E \times [0,T].
\]
In this setting, we summarize the properties of $[v]_h$ and $[v]_{\bar{h}}$ displayed below, whose detailed proof can be seen in the literatures~\cite[Lemma 2.2]{KL} and~\cite[Appendix B]{BDM}.
\begin{lem}\label{mollification lemma}
Assume that $v\in L^1(E_T)$ and $p \in [1,\infty)$, where $E_T:=E \times (0,T)$. Then the mollifications $[v]_{h}$ and $[v]_{\bar{h}}$ have the following properties:
\begin{enumerate}[\normalfont(i)]
\item If $v \in L^p(E_T)$, then $[v]_h \in L^p(E_T)$ and the inequality holds true:
\[
\|[v]_h\|_{L^p(E_T)} \leq \|v\|_{L^p(E_T)}.
\]
Furthermore, 
\[
[v]_h \to v \quad \textrm{strongly in}\,\,\, L^p(E_T)\quad \textrm{as} \,\,\,h \searrow 0. 
\]
A same statement for $[v]_{\bar{h}}$ holds true.
\item If $v \in L^p(E_T)$, then $[v]_h$ and $[v]_{\bar{h}}$ have weak time derivatives being in $L^p(E_T)$ and solve the ODE:
\begin{align*}
\partial_t[v]_h=-\frac{[v]_h-v}{h} \quad;\quad \partial_t[v]_{\bar{h}}=\frac{[v]_{\bar{h}}-v}{h}.
\end{align*}
\item If $v \in L^p(0,T\,;W^{1,p}_0(E))$, then $[v]_h \in L^p(0,T\,;W^{1,p}_0(E))$. Furthermore,  there holds that
$[v]_h \to v$ strongly in $L^p(0,T\,;W^{1,p}_0(E))$ as $h \searrow 0$. The same implication for $[v]_{\bar{h}}$ holds true.
\item If $v \in L^p(0,T\,; L^p(E))$, then $[v]_h$ and $[v]_{\bar{h}}$ belong to $C([0,T]\,; L^p(E))$.
\item If $v \in L^\infty(0,T\,; L^2(E))$, then $[v]_h \in C([0,T]\,; L^2(E))$. Moreover we have
\[
[v]_h \to v \quad \textrm{in} \quad C([0,T]\,;L^2(E)) \quad \textrm{as}\,\,\, h\searrow 0.
\]
\end{enumerate}
\end{lem}
Hereafter, we shall apply Lemma~\ref{mollification lemma} with $E=\Omega$ or $E=\Omega \times \Omega$ on many times.
\medskip

Using Lemma~\ref{mollification lemma}, we deduce the following lemma.
\begin{lem}\label{mollification lemma2}
With $p \geq 1$, suppose that $v \in L^\infty(0,T\,;L^p(\Omega))$ and abbreviate $w:=|v|^{\frac{p-2}{2}}v$. Set
\[
\langle v\rangle_h:=
\begin{cases}
\left|[w]_h\right|^{\frac{2}{p}-1}[w]_h\quad &\textrm{if}\,\,\,v \neq 0,\\
0 \quad &\textrm{if}\,\,\,v \neq 0.
\end{cases}
\]
Then, $\langle v\rangle_h \in C([0,T]\,;L^p(\Omega))$ and moreover $\langle v \rangle_h \to v$ in $C([0,T]\,;L^p(\Omega))$ as $h \searrow 0$.
\end{lem}
\begin{proof}
First of all, by the assumption that $v \in L^\infty(0,T\,;L^p(\Omega))$, $w:=|v|^{\frac{p-2}{2}}v$ belongs to $L^\infty(0,T\,;L^2(\Omega))$. Thus, Lemma~\ref{mollification lemma}-(v) yields that
\begin{equation}\label{mollification eq.1}
[w]_h \in C([0,T]\,;L^2(\Omega))
\end{equation}
and
\begin{equation}\label{mollification eq.2}
[w]_h \to w \quad\textrm{in}\,\,\,C([0,T]\,;L^2(\Omega))\quad \textrm{as}\,\,\, h\,\downarrow\,0.
\end{equation}
We shall show the first implication. For this, we now distinguish two cases between $p \geq 2$ and $1\leq p<2$. As we are considering the case $p \geq 2$, applying Lemma~\ref{Algs} with $\beta=\frac{p}{2}$, we infer that for every $0 \leq t_1<t_2 \leq T$
\begin{align*}
\big|\langle v \rangle_h (t_1)-\langle v \rangle_h(t_2)\big|^{\frac{p}{2}} \leq c \Big||\langle v \rangle_h|^{\frac{p-2}{2}}\langle v \rangle_h(t_1)-|\langle v \rangle_h|^{\frac{p-2}{2}}\langle v \rangle_h(t_2)\Big|=c\big|[w]_h(t_1)-[w]_h(t_2)\big|
\end{align*}
and therefore, this together with~\eqref{mollification eq.2} implies that
\[
\big\|\langle v \rangle_h (t_1)-\langle v \rangle_h(t_2)\big\|_{L^p(\Omega)}^p \leq c \big\|[w]_h(t_1)-[w]_h(t_2)\big\|_{L^2(\Omega)}^2 \to 0
\]
as $t_2-t_1 \to 0$. in the remaining  case $1\leq p<2$, we again use Lemma~\ref{Algs} with $\beta=\frac{p}{2}$ to get
\begin{align*}
\Big(|\langle v \rangle_h(t_1)|+|\langle v \rangle_h(t_2)|\Big)^{\frac{p-2}{2}}\big|\langle v \rangle_h(t_1)-\langle v \rangle_h(t_2)\big| &\leq c \Big||\langle v \rangle_h|^{\frac{p-2}{2}}\langle v \rangle_h(t_1)-|\langle v \rangle_h|^{\frac{p-2}{2}}\langle v \rangle_h(t_2)\Big| \\[3mm]
&=c\big|[w]_h(t_1)-[w]_h(t_2)\big|
\end{align*}
and thus,
\[
\int_\Omega \Big(|\langle v \rangle_h(t_1)|+|\langle v \rangle_h(t_2)|\Big)^{p-2}\big|\langle v \rangle_h(t_1)-\langle v \rangle_h(t_2)\big|^2\dx \leq c\big\|[w]_h(t_1)-[w]_h(t_2)\big\|_{L^2(\Omega)}^2 \to 0
\]
as $t_2-t_1 \to 0$. Using this, ~\eqref{mollification eq.1} and H\"{o}lder's inequality with the exponent $\left(\frac{2}{p},\frac{2}{2-p}\right)$, we infer that
\begin{align*}
\int_\Omega |\langle v \rangle_h(t_1)-\langle v \rangle_h(t_2)|^p\dx &\leq \left[\int_\Omega \Big(|\langle v \rangle_h(t_1)|+|\langle v \rangle_h(t_2)|\Big)^{p-2}\big|\langle v \rangle_h(t_1)-\langle v \rangle_h(t_2)\big|^2\dx\right]^{\frac{p}{2}} \\[3mm]
&\quad \quad \quad \times \left[\int_\Omega \Big(|\langle v \rangle_h(t_1)|+|\langle v \rangle_h(t_2)|\Big)^p\dx\right]^{\frac{2-p}{2}} \\[3mm]
&\leq c\big\|[w]_h(t_1)-[w]_h(t_2)\big\|_{L^2(\Omega)}^2 \Big(\|[w]_h(t_1)\|_{L^2(\Omega)}^2+\|[w]_h(t_2)\|_{L^p(\Omega)}^p\Big)^{\frac{2-p}{2}} \\[3mm]
%&\leq c\big\|[w]_h(t_1)-[w]_h(t_2)\big\|_{L^2(\Omega)}^2 \Big(2\|[w]_h(t)-w(t)\|_{L^2(\Omega)}^2+3\|v(t)\|_{L^p(\Omega)}^p\Big)^{\frac{2-p}{2}} \\[3mm]
&\to 0
\end{align*}
as $t_2-t_1 \to 0$, where in the last line we used 
\[
\|\langle v \rangle_h(t_i)\|_{L^p(\Omega)}^p=\|[w]_h(t_i)\|_{L^2(\Omega)}^2 \quad \textrm{for} \,\, i=1,2.
\]
Therefore we conclude that, for every $p \geq 1$,
\[
\|\langle v \rangle_h(t_1)-\langle v \rangle_h(t_2)\|_{L^p(\Omega)} \to 0 \quad \textrm{as}\,\,\,t_2-t_1\to 0,
\]
as desired. Similarly as above, by the use of~\eqref{mollification eq.2} we deduce that,
\begin{align*}
&\|\langle v \rangle_h(t)-v(t)\|_{L^p(\Omega)}^p \\[3mm]
&
\leq \begin{cases}
c \big\|[w]_h(t)-w(t)\big\|_{L^2(\Omega)}^2\quad &\textrm{for}\,\,\,p\geq 2 \\[2mm]
c\big\|[w]_h(t)-w(t)\big\|_{L^2(\Omega)}^2\Big(\|[w]_h(t)\|_{L^2(\Omega)}^2+\|w(t)\|_{L^2(\Omega)}^2\Big)^{\frac{2-p}{2}} \quad &\textrm{for}\,\,\,1\leq p< 2
\end{cases}
\\[3mm]
&\to 0
\end{align*}
as $h\searrow 0$, showing the second implication and therefore, the proof is complete.
\end{proof}
%%%%%%%%%%%%%%%%%%%%%%%%%%%%%%%%
%%%%%%%%%%%%%%%%%%%%%%%%%%%%%%%
%%%%%%%%%%%%%%%%%%%%%%%%%%%%%%

\subsection{Mollified weak formulation~\eqref{D2}}
We next list the mollified version of~\eqref{D2} in Definition~\ref{weak sol}, whose proof is based on the Fubini theorem for the double integral, see~\cite[Lemma 2.10]{Nakamura} for the details. This enables us to choose various testing functions in rigorous way.   
\begin{lem}
Let $u$ be a weak sub(super)-solution to~\eqref{maineq} in the sense of Definition~\ref{weak sol}. Then for every nonnegative $\varphi \in \mathcal{T}$, there holds that
\begin{align}\label{D2'}
&\iint_{\Omega_T}\left(\partial_t[|u|^{p-2}u]_{h}\,\varphi+\left[\mathbf{A}\left(x,t,u,Du\right)\right]_h\cdot D\varphi\right)\dxt \notag \\[2mm]
&\quad +\int_0^T\iint_{\bR^n \times \bR^n} \left[U(x,y,t)K(x,y,t)\right]_h\left(\varphi(x,t)-\varphi(y,t)\right)\dxyt \notag\\[2mm]
\begin{split}
&\quad\,\leq\\[-2mm]
&\quad (\geq)
\end{split}
\int_\Omega |u|^{p-2}u(0)\left(\dfrac{1}{h}\int_0^Te^{\frac{s}{h}}\varphi(x,s)\ds\right)\dx.
\end{align}
\end{lem}
%
%
%

%%%%%%%%%%%%%%%%%%%%%%%%%
\section{Quantitative estimates for supersolutions}\label{Sect. 3}

In this section we give quantitative estimates for supersolutions. This section is twofold. First we prove the Reverse H\"{o}lder inequality for supersolutions, whose proof is based on the Moser's iteration scheme. Besides, we shall derive the log-type estimate, that will be key ingredient to prove Theorem~\ref{mainthm}.

\subsection{Reverse H\"{o}lder's inequality}\label{Sect. 3.1}
In order to prove the Reverse H\"{o}lder inequality (Proposition~\ref{reverse Holder}), we need some preliminary results.%

The first step is to establish the following Caccioppoli type inequality. 
\begin{lem}[Caccioppoli type estimate for supersolutions]\label{Caccioppoli1}
Let $u$ be a weak supersolution to~\eqref{maineq} fulfilling $u \geq m>0$ in $\bR^n \times (t_0-\rho^p,t_0)$. With $p >1$ let $\varepsilon \in (0,p-1)$ and set $\alpha:=p-1-\varepsilon$. Then
\begin{align*}
&\sup_{t \in (t_0-\rho^p,t_0)}\int_{B_\rho(x_0) \times \{t\}}u^\alpha\varphi^p\dx+\iint_{Q^-_\rho(z_0)}|Du|^pu^{-\varepsilon-1}\varphi^p\dxt\\[4mm]
&\quad \quad \quad \quad \quad +\int_{t_0-\rho^p}^{t_0}\iint_{B_\rho(x_0)\times B_\rho(x_0)}\dfrac{\big|u(x,t)^{\frac{\alpha}{p}}\varphi(x,t)-u(y,t)^{\frac{\alpha}{p}}\varphi(y,t)\big|^p}{|x-y|^{n+sp}}\dxyt\\[4mm]
&\leq c\iint_{Q^-_\rho(z_0)}u^\alpha \varphi^{p-1}|\varphi_t|\dxt+c\iint_{Q^-_\rho(z_0)}u^\alpha |D\varphi|^p\dxt\\[4mm]
&\quad \quad +c\int_{t_0-\rho^p}^{t_0}\iint_{B_\rho(x_0)\times B_\rho(x_0)}\dfrac{\big(u(x,t)^{\alpha}+u(y,t)^{\alpha}\big)\big|\varphi(x,t)-\varphi(y,t)\big|^p}{|x-y|^{n+sp}}\dxyt \\[4mm]
&\quad \quad +c \left(\sup_{x \in \supp \,\varphi(\cdot, t)}\int_{\bR^n \setminus B_\rho(x_0)}\dfrac{\dy}{|x-y|^{n+sp}}\right)\iint_{Q^-_\rho(z_0)}u^\alpha\varphi^p\dxt
\end{align*}
holds whenever nonnegative $\varphi \in C_0^\infty(Q^-_\rho(z_0))$, where the constant $c \equiv c(p,c_0,c_1,\Lambda,\varepsilon)$ blows up as $\varepsilon \searrow 0$ and $\varepsilon \nearrow  p-1$.
\end{lem} 
%%%%%%%%%%%%%%%%%%%%%%%%%%%%%%%
For the readability, we shall postpone the proof to Appendix~\ref{Appendix A}.

\medskip

The subsequent lemma is the starting point of Moser's iteration scheme that will be argued later in the proof of Proposition~\ref{reverse Holder}.
\begin{lem}\label{reverse lemma}
Let $u$ be a weak supersolution to~\eqref{maineq} fulfilling $u \geq m>0$ in $\bR^n \times (t_0-\rho^p,t_0)$. Let $\varepsilon \in (0,p-1)$ and set $\alpha:=p-1-\varepsilon$ and $\kappa:=\frac{n+p}{n}$. Then for any concentric $Q_r^-(z_0) \subset Q^-_\rho(z_0) \Subset \Omega_T$ the quantitative estimate 
\begin{align*}
\left(\fiint_{Q^-_r(z_0)}u^{\alpha \kappa}\dxt\right)^{\frac{1}{\alpha \kappa}}\leq c^{\frac{1}{\alpha \kappa}}\left[\left(\frac{\rho}{r}\right)^n\left(\frac{\rho}{\rho-r}\right)^{n+sp}(1+\rho^{(1-s)p})\right]^{\frac{1}{\alpha}}\left(\fiint_{Q^-_\rho(z_0)}u^\alpha\dxt\right)^{\frac{1}{\alpha}}
\end{align*}
holds true, where $c=c(n,s,p,\varepsilon)$ blows up as $\varepsilon \searrow 0$ or $\varepsilon \nearrow p-1$.
\end{lem}
\begin{proof}
Take a cut-off function $\varphi \in C^\infty(Q^-_\rho(z_0))$ satisfying%
\[
\begin{cases}
\supp \,\varphi \subset Q^-_{\frac{\rho+r}{2}}(z_0), \quad 0 \leq \varphi \leq 1 \,\,\,\textrm{in}\,\,\,Q^-_{\rho}(z_0),\quad \varphi \equiv 1\,\,\,\textrm{on}\,\,\,Q^-_r(z_0)\,;\\[2mm]
|D\varphi|\leq \dfrac{c}{\rho-r},\quad |\varphi_t|\ \leq \dfrac{c}{(\rho-r)^p}.
\end{cases}
\]
We now observe that
\[
\fiint_{Q^-_r(z_0)}(u^{\frac{\alpha}{p}})^{p\kappa}\dxt=\fiint_{Q^-_r(z_0)}(u^{\frac{\alpha}{p}}\varphi)^{p\kappa}\dxt \leq \left(\frac{\rho}{r}\right)^{n+p}\fiint_{Q^-_\rho(z_0)}(u^{\frac{\alpha}{p}}\varphi)^{p\kappa}\dxt. 
\]
Applying the Gagliardo-Nirenberg inequality~\eqref{GNeq.} in Lemma~\ref{GN} to $v\equiv u^{\frac{\alpha}{p}}\varphi$ and the Caccioppoli type estimate (Lemma~\ref{Caccioppoli1}) yield that
\begin{align*}
&\fiint_{Q^-_\rho(z_0)}(u^{\frac{\alpha}{p}}\varphi)^{p\kappa}\dxt \\[3mm]
&\leq c\rho^p \left(\sup_{t \in (t_0-\rho^p,t_0)}\dashint_{B_\rho(x_0)}(u^{\frac{\alpha}{p}}\varphi)^p\dx \right)^{\frac{p}{n}}\fiint_{Q^-_\rho(z_0)}\left(|D(u^{\frac{\alpha}{p}}\varphi)|^p+\left|\frac{u^{\frac{\alpha}{p}}\varphi}{\rho}\right|^p\right)\dxt \\[4mm]
&\leq \frac{c}{\rho^{n+p}}\left(\sup_{t \in (t_0-\rho^p,t_0)}\int_{B_\rho(x_0)}u^\alpha\varphi^p\dx \right)^{\frac{p}{n}}\iint_{Q^-_\rho(z_0)}\left(u^\alpha|D\varphi|^p+\alpha^pu^{\alpha-p}|Du|^p\varphi^p+\frac{u^\alpha \varphi^p}{\rho^p}\right)\dxt \\[4mm]
&\leq \frac{c}{\rho^{n+p}}\left[\iint_{Q^-_\rho(z_0)}\left(u^\alpha\varphi^{p-1}|\varphi_t|+u^\alpha|D\varphi|^p+\frac{u^\alpha \varphi^p}{\rho^p}\right)\dxt \right.\\[4mm]
&\quad \quad \quad \quad \left.+\int_{t_0-\rho^p}^{t_0}\iint_{B_\rho(x_0)\times B_\rho(x_0)}\frac{\left(u(x,t)^\alpha+u(y,t)^\alpha\right)|\varphi(x,t)-\varphi(y,t)|^p}{|x-y|^{n+sp}}\dxyt \right.\\[4mm]
&\quad \quad \quad \quad \left.+\left(\sup_{x \in \supp\,\varphi(\cdot, t)}\int_{\bR^n \setminus B_\rho(x_0)}\dfrac{\dy}{|x-y|^{n+sp}}\right)\iint_{Q^-_\rho(z_0)}u^\alpha \varphi^p\dxt \right]^\kappa
 \\[4mm]
&\leq \frac{c}{\rho^{n+p}}\left[\iint_{Q^-_\rho(z_0)}\frac{1}{(\rho-r)^p}u^\alpha\dxt \right.\\[4mm]
&\quad \quad \quad \quad \left.+\int_{t_0-\rho^p}^{t_0}\iint_{B_\rho(x_0)\times B_\rho(x_0)}\frac{\left(u(x,t)^\alpha+u(y,t)^\alpha\right)}{|x-y|^{n+(1-s)p}}|D\varphi|^p\dxyt \right.\\[4mm]
&\quad \quad \quad \quad \left.+\left(\sup_{x \in B_{\frac{\rho+r}{2}}(x_0)}\int_{\bR^n \setminus B_\rho(x_0)}\dfrac{\dy}{|x-y|^{n+sp}}\right)\iint_{Q^-_\rho(z_0)}u^\alpha \varphi^p\dxt \right]^\kappa.
\end{align*}
Thus, combining this with the preceding estimate gives
\begin{align}\label{reverse eq.1}
\left(\fiint_{Q^-_r(z_0)}u^{\alpha \kappa}\dxt\right)^{\frac{1}{\alpha \kappa}} 
&\leq \frac{c}{r^{n+p}}\left[\frac{1}{(\rho-r)^p}\iint_{Q^-_\rho(z_0)}u^\alpha\dxt \right.\notag\\[4mm]
&\quad \quad \quad \quad\left.+\frac{1}{(\rho-r)^p}\int_{t_0-\rho^p}^{t_0}\iint_{B_\rho(x_0)\times B_\rho(x_0)}\frac{\left(u(x,t)^\alpha+u(y,t)^\alpha\right)}{|x-y|^{n+(1-s)p}}\dxyt \right.\notag \\[4mm]
&\quad \quad \quad \quad\left.+\left(\sup_{x \in B_{\frac{\rho+r}{2}}(x_0)}\int_{\bR^n \setminus B_\rho(x_0)}\dfrac{\dy}{|x-y|^{n+sp}}\right)\iint_{Q^-_\rho(z_0)}u^\alpha\dxt \right]^\kappa \notag\\[4mm]
&=:\frac{c}{r^{n+p}}\Big[\mathbf{I}+\mathbf{II}+\mathbf{III}\Big]^\kappa,
\end{align}
where the definition of $\mathbf{I}$-$\mathbf{III}$ are clear from the context and note that, the constant $c$ appearing on the final line has the singularity at $\varepsilon=0$ and $\varepsilon=p-1$. 

We will estimate $\mathbf{I}$-$\mathbf{III}$ separately. Rearranging the term $\mathbf{I}$ yields
\begin{equation}\label{reverse eq.2}
\mathbf{I}=\frac{1}{(\rho-r)^p}\rho^{n+p}\fiint_{Q^-_\rho(z_0)}u^\alpha\dxt.
\end{equation}
By exchanging the role of $u(x,t)$ and $u(y,t)$, we have
\begin{align}\label{reverse eq.3}
\mathbf{II}&=\frac{2}{(\rho-r)^p}\int_{t_0-\rho^p}^{t_0}\iint_{B_\rho(x_0)\times B_\rho(x_0)}\frac{u^\alpha(x,t)}{|x-y|^{n-(1-s)p}}\dxyt \notag\\[3mm]
&=\frac{2}{(\rho-r)^p}\iint_{Q^-_\rho(z_0)}u^\alpha(x,t)\left(\int_{B_\rho(x_0)}\frac{\dy}{|x-y|^{n-(1-s)p}}\right)\dxt \notag\\[3mm]
&=\frac{c(n,s,p)}{(\rho-r)^p}\rho^{(1-s)p+n+p}\fiint_{Q^-_\rho(z_0)}u^\alpha\dxt, 
\end{align}
where, in the second line, we computed that
\[
\displaystyle \int_{B_\rho(x_0)}\frac{\dy}{|x-y|^{n-(1-s)p}}=|B_\rho|\frac{\rho^{(1-s)p}}{(1-s)p}. 
\]
We shall handle the term $\mathbf{III}$. Since
\[
\frac{|y-x_0|}{|y-x|}\leq \frac{|y-x|+|x-x_0|}{|y-x|}=1+\frac{|x-x_0|}{|y-x|}
\]
and
\[
|y-x| \geq |y-x_0|-|x_0-x| \geq \rho-\frac{\rho+r}{2}=\frac{\rho-r}{2}
\]
holds whenever $x \in B_{\frac{\rho+r}{2}}(x_0)$ and $y \in \bR^n \setminus B_\rho(x_0)$, we have 
\[
\dfrac{|y-x_0|}{|y-x|} \leq \dfrac{2\rho}{\rho-r}\qquad \forall x \in B_{\frac{\rho+r}{2}}(x_0),\quad \forall y \in \bR^n \setminus B_\rho(x_0).
\]
Thus,
\begin{align}\label{reverse eq.4}
\mathbf{III} &\leq c\left(\frac{\rho}{\rho-r}\right)^{n+sp}\left( \sup_{x \in B_{\frac{\rho+r}{2}}(x_0)}\int_{\bR^n \setminus B_\rho(x_0)}\frac{\dy}{|y-x_0|^{n+sp}}\right)\iint_{Q^-_\rho(z_0)}u^\alpha\dxt \notag\\[4mm]
&=c\left(\frac{\rho}{\rho-r}\right)^{n+sp}\rho^{-sp}\iint_{Q^-_\rho(z_0)}u^\alpha\dxt \notag\\[4mm]
&=c\left(\frac{\rho}{\rho-r}\right)^{n+sp}\rho^{n+(1-s)p}\fiint_{Q^-_\rho(z_0)}u^\alpha\dxt,
\end{align}
where, in the second line we used that $\displaystyle \int_{\bR^n \setminus B_\rho(x_0)}\frac{\dy}{|y-x_0|^{n+sp}}=\frac{c(n)}{sp}\rho^{-sp}$. Noticing that $r^{n+p}=r^{n\kappa}$ and inserting the preceding estimates~\eqref{reverse eq.2}--\eqref{reverse eq.4} back to~\eqref{reverse eq.1} concludes that
\begin{align*}
&\fiint_{Q^-_\rho(z_0)}(u^{\frac{\alpha}{p}}\varphi)^{p\kappa}\dxt \\[4mm]
&\leq \frac{c}{r^{n+p}}\left[\left(\frac{\rho}{\rho-r}\right)^p\rho^n\left(1+\rho^{(1-s)p}+\left(\frac{\rho}{\rho-r}\right)^{n+(s-1)p}\rho^{(1-s)p}\right)\fiint_{Q^-_\rho(z_0)}u^\alpha\dxt \right]^\kappa \\[4mm]
&\leq c\left(\frac{\rho}{r}\right)^{n\kappa}\left[\left(\frac{\rho}{\rho-r}\right)^{n+sp}\left(1+\rho^{(1-s)p}\right)\fiint_{Q^-_\rho(z_0)}u^\alpha\dxt \right]^\kappa,
\end{align*}
which proves the claim.
\end{proof}

We are ready to prove the Reverse H\"{o}lder's inequality about supersolutions.
\begin{prop}[Reverse H\"{o}lder's inequality]\label{reverse Holder}
Assume that $u$ is a weak supersolution to~\eqref{maineq} fulfilling $u\geq m>0$ in $\bR^n \times (t_0-\rho^p,t_0)$. Fix a parameter $\sigma_0 \in (0,1)$. Then, for any $Q_\rho(z_0) \Subset \Omega_T$, there exists a positive constant $C=C(n,s,p,c_0,c_1,\Lambda,\sigma_0)$ such that
\[
\left(\fiint_{Q^-_{\sigma\rho}(z_0)}u^q\dxt\right)^{\frac{1}{q}} \leq C\left(\frac{1+\rho^{(1-s)p}}{(\tau-\sigma)^{n+sp}}\right)^{\frac{(n+p)^2}{np\gamma}} \left(\fiint_{Q^-_{\tau\rho}(z_0)}u^\gamma\dxt\right)^{\frac{1}{\gamma}}
\]
holds whenever $\sigma_0 \leq \sigma<\tau<1$ and $0<\gamma<q<\frac{n+p}{n}(p-1)$.
\end{prop}
\begin{proof}
First of all, for every $q,\,\gamma$ satisfying $q>\gamma$, we choose $k \in \mathbb{N}$ so that
\[
\frac{\log\tfrac{q}{\gamma}}{\log \kappa} \leq k \leq \frac{\log\tfrac{q}{\gamma}}{\log \kappa} +1
\]
and for $q_0<\gamma$, we set $q=q_k=q_0\kappa^k$. Next, for $i=0,1,\ldots, k$ let us define
\[
\rho_i:=\left(\tau-(\tau-\sigma)\frac{1-2^{-i}}{1-2^{-k}}\right)\rho, \qquad Q_i:=Q^-_{\rho_i}(z_0)=B_{\rho_i}(x_0)\times (t_0-\rho_i^p,t_0),
\]
and therefore the following inclusions trivially hold:
\[
\rho_0=\tau\rho \geq \cdots \geq \rho_i \geq \rho_k=\sigma \rho,\qquad Q_0=Q^-_{\tau\rho} \supset \cdots \supset Q_i \supset \cdots \supset Q_k=Q^-_{\sigma\rho}.
\]
Applying Lemma~\ref{reverse lemma} with $r=\rho_{i+1}$, $\rho=\rho_i$ and
\[
q_{i+1}=\alpha \kappa, \quad q_i=\alpha\quad (\Longrightarrow  q_i=q_0\kappa^i),
\]
we infer that, for $i=0,1,\ldots, k-1$,
\begin{align}\label{reverse eq.a}
\left(\fiint_{Q_{i+1}}u^{q_{i+1}}\dxt\right)^{\frac{1}{q_{i+1}}}\leq c^{\frac{1}{\alpha \kappa}}\left[\left(\frac{\rho_i}{\rho_{i+1}}\right)^n\left(\frac{\rho_i}{\rho_i-\rho_{i+1}}\right)^{n+sp}(1+\rho_i^{(1-s)p})\right]^{\frac{1}{q_i}}\left(\fiint_{Q_i}u^{q_i}\dxt\right)^{\frac{1}{q_i}}.
\end{align}
We observe that
\[
\frac{\rho_i}{\rho_i-\rho_{i+1}}=\frac{\tau-(\tau-\sigma)\frac{1-2^{-i}}{1-2^{-k}}}{(\tau-\sigma)\frac{2^{-(i+1)}}{1-2^{-k}}}\leq \frac{2^{i+1}\tau}{\tau-\sigma}
\]
and $\rho_i^{(1-s)p} \leq \rho_0^{(1-s)p}=\rho^{(1-s)p}$, which junction with~\eqref{reverse eq.a} yields%
\begin{align}\label{reverse eq.b}
\left(\fiint_{Q_{i+1}}u^{q_{i+1}}\dxt\right)^{\frac{1}{q_{i+1}}}\leq c^{\frac{1}{q_{i+1}}}\left[\left(\frac{\rho_i}{\rho_{i+1}}\right)^n\left(\frac{2^{i+1}\tau}{\tau-\sigma}\right)^{n+sp}(1+\rho^{(1-s)p})\right]^{\frac{1}{q_i}}\left(\fiint_{Q_i}u^{q_i}\dxt\right)^{\frac{1}{q_i}}.
\end{align}
Note that, this estimate~\eqref{reverse eq.b} is only valid for $0<q_i<p-1$\footnote{Since $\{q_i\}_{i=1}^k$ is increasing sequence, it is enough to consider that $0<q=q_k=q_{k-1}\kappa<(p-1)\kappa$.} since the constant $c$ possesses the singularity at $q_i=0$ or $q_i=p-1$. For the rest of the argument, let us denote in short, for $i=0,1,\ldots, k$
\[
Y_i:=\left(\fiint_{Q_i}u^{q_i}\dxt\right)^{\frac{1}{q_i}}.
\]
Iterating~\eqref{reverse eq.b}, for $i=0,1,\ldots , k$, gives that
\begin{align}\label{reverse eq.c}
\left(\fiint_{Q_{\sigma\rho}}u^{q}\dxt\right)^{\frac{1}{q}}=Y_k &\leq c^{\frac{1}{q_k}}\left(\frac{\rho_{k-1}}{\rho_k}\right)^{\frac{1}{q_k}}2^{(n+sp)\frac{k-1}{q_{k-1}}}\left(\frac{1+\rho^{(1-s)p}}{(\tau-\sigma)^{n+sp}}\right)^{\frac{1}{q_{k-1}}}Y_{k-1} \notag\\
& \,\,\, \vdots \notag\\
&\leq c^{S(k)}2^{(n+sp)T(k)}M(k)\left(\frac{1+\rho^{(1-s)p}}{(\tau-\sigma)^{n+sp}}\right)^{S(k)}Y_0,
\end{align}
where
\begin{align*}
S(k)&:=\sum_{i=0}^k\frac{1}{q_i},\quad T(k):=\sum_{i=0}^k\frac{i}{q_{i}} \\[2mm]
M(k)&:=\prod_{i=1}^k \left(\frac{\rho_{i-1}}{\rho_i}\right)^{\frac{1}{q_i}}.
\end{align*}
We estimate the above quantities $S(k)$, $T(k)$ and $M(k)$. Since 
\[
\gamma \kappa^{k-1} \leq q_0\kappa^k=q_k \iff q_0 \geq \gamma/\kappa, 
\]
a straightforward computation implies that
\begin{equation*}
S(k)=\sum_{i=0}^k\frac{1}{q_0\kappa^i}=\frac{\kappa}{q_0(\kappa-1)}\left[1-\left(\frac{1}{\kappa}\right)^k\right]\leq \frac{\kappa^2}{\gamma (\kappa-1)}
\end{equation*}
and 
\begin{align*}
T(k)=\sum_{i=0}^k\frac{i}{q_0\kappa^{i}}=-\frac{k}{q_0(\kappa-1)\kappa^{k-1}}+\frac{\kappa}{q_0(\kappa-1)^2}\left[1-\left(\frac{1}{\kappa}\right)^k\right]\leq \frac{\kappa^2}{\gamma (\kappa-1)^2}.
\end{align*}
Also, we estimate that
\begin{align*}
M(k)&=\exp\left(\sum_{i=1}^k\frac{1}{q_0\kappa^i}\log\left(\frac{\rho_{i-1}}{\rho_i}\right) \right)\leq \sum_{i=1}^k\frac{1}{q_0\kappa}\left(\log \rho_{i-1}-\log \rho_i\right) \\[2mm]
&\leq \exp\left(\frac{1}{\gamma}\log\left(\frac{1}{\sigma}\right)\right)\leq \left(\frac{1}{\sigma_0}\right)^{\frac{1}{\gamma}}.
\end{align*}
Therefore, 
\[
C_{\textrm{prod}}(k):=c^{S(k)}2^{(n+sp)T(k)}M(k)
\]
is uniformly bounded on $k$ and estimated as
\[
C_{\textrm{prod}}(k)\leq c^{\frac{\kappa^2}{\gamma (\kappa-1)}}2^{(n+sp)\frac{\kappa^2}{\gamma (\kappa-1)^2}}\left(\frac{1}{\sigma_0}\right)^{\frac{1}{\gamma}}=:C,
\]
where the constant $C$ depends on $n,s,p,c_0,c_1,\Lambda, \gamma$ and $\sigma_0$. Combining this with~\eqref{reverse eq.c} and applying H\"{o}lder's inequality to $Y_0$, we finally deduce that, for every $\sigma_0 \leq \sigma<1$ and $\gamma<q<\kappa(p-1)=\frac{n+p}{n}(p-1)$,
\[
\left(\fiint_{Q^-_{\sigma\rho}}u^{q}\dxt\right)^{\frac{1}{q}} \leq C\left(\frac{1+\rho^{(1-s)p}}{(\tau-\sigma)^{n+sp}}\right)^{ \frac{\kappa^2}{\gamma (\kappa-1)}}\left(\fiint_{Q^-_{\tau \rho}}u^{\gamma}\dxt\right)^{\frac{1}{\gamma}},
\]
as desired.
\end{proof}
\subsection{Log-type estimates}\label{Sect. 3.2}
To begin, we prove the log-type Caccioppoli estimate.
\begin{lem}[Log-type Caccioppoli estimate]\label{log-type Caccioppoli}
Let $0 \leq t_1<t_2 \leq T$ and suppose that $u$ is a weak supersolution to~\eqref{maineq} such that $u \geq m >0$ in $\bR^n \times (t_1,t_2)$. Then, there exists a positive constant $c\equiv c(p,c_0,c_1)$ such that, for every nonnegative $\varphi \in C^\infty_0(\Omega_{t_1,t_2})$, the following quantitative estimate holds true:
\begin{align*}
&\iint_{\Omega_{t_1,t_2}}\varphi^p\left|D(\log u)\right|^p\dxt-\left[\int_\Omega \varphi^p \log u\dx\right]_{t=t_1}^{t_2} \\[4mm]
&\leq c\iint_{\Omega_{t_1,t_2}}|D\varphi|^p\dxt+c\iint_{\Omega_{t_1,t_2}}\varphi^{p-1}|\varphi_t|\,|\log u|\dxt \\[4mm]
&\quad +c\int_{t_1}^{t_2}\iint_{\bR^n \times \bR^n}U(x,y,t)K(x,y,t)\left(u(x,t)^{1-p}\varphi(x,t)^p-u(y,t)^{1-p}\varphi(y,t)^p\right)\dxyt,
\end{align*}
where $\Omega_{t_1,t_2}:=\Omega \times (t_1,t_2)$. 
\end{lem}
\begin{proof}
First of all,  for every  $0\leq t_1<t_2\leq T$ and $\varepsilon>0$ small enough, we define the following Lipschitz cut-off function:
 \begin{equation*}
\chi=\chi_\varepsilon(t):=
\begin{cases}
0, \quad &t \in [0, t_1),\\
\frac{1}{\varepsilon}(t-t_1), \quad &t \in [t_1, t_1+\varepsilon),\\
1, \quad &t \in [t_1+\varepsilon, t_2-\varepsilon],\\
-\frac{1}{\varepsilon}(t-t_2), \quad &t \in (t_2-\varepsilon, t_2], \\
0, \quad &t \in (t_2, T].
\end{cases}
\end{equation*}
Choosing the testing function in~\eqref{D2'} as
\begin{equation*}\label{cut-off}
\varphi(x,t)=\chi_\varepsilon(t)u(x,t)^{1-p}\varphi(x,t)^p\quad \textrm{with}\,\,\,0 \leq \varphi \in C^\infty_0(\Omega_T)
\end{equation*}
gives the estimates below respectively: The parabolic term is split into two terms
\begin{align*}
\iint_{\Omega_T}\partial_t[|u|^{p-2}u]_h\chi_\varepsilon u^{1-p}\varphi^p\dxt 
&=\iint_{\Omega_T}\chi_\varepsilon\varphi^p\partial_t[u^{p-1}]_h\left(u^{1-p}-[u^{p-1}]_h^{-1}\right)\dxt \\[2mm]
&\quad +\iint_{\Omega_T}\chi_\varepsilon\varphi^p\partial_t[u^{p-1}]_h[u^{p-1}]_h^{-1}\dxt\\[2mm]
&=:\mathbf{I}_1+\mathbf{I}_2.
\end{align*}
Apply Lemma~\ref{mollification lemma}-(ii), we have 
\begin{align*}
\mathbf{I}_1&=\iint_{\Omega_T}\chi_\varepsilon\varphi^p\frac{u^{p-1}-[u^{p-1}]_h}{h}\cdot \frac{[u^{p-1}]_h-u^{p-1}}{[u^{p-1}]_hu^{p-1}}\dxt  \\[3mm]
&=-\frac{1}{h}\iint_{\Omega_T}\chi_\varepsilon\varphi^p\frac{\left(u^{p-1}-[u^{p-1}]_h\right)^2}{[u^{p-1}]_hu^{p-1}}\dxt \leq 0.
\end{align*}
By integration by parts, we infer that
\begin{align*}
\mathbf{I}_2&=\iint_{\Omega_T}\chi_\varepsilon \varphi^p\partial_t\left(\log [u^{p-1}]_h\right)\dxt\\[3mm]
&=-\iint_{\Omega_T}\left(\chi_\varepsilon^\prime \varphi^p+p\varphi^{p-1}\varphi^\prime \chi_\varepsilon\right)\log [u^{p-1}]_h\dxt \\[3mm]
&=-\dashint_{t_1}^{t_1+\varepsilon} \int_\Omega \varphi^p\log[u^{p-1}]_h\dxt+\dashint_{t_2-\varepsilon}^{t_2} \int_\Omega \varphi^p\log[u^{p-1}]_h\dxt \\[3mm]
&\quad \quad \quad \quad \quad -p\iint_{\Omega_{t_1,t_2}}\varphi^{p-1}\varphi^\prime \chi_\varepsilon \log[u^{p-1}]_h\dxt.
\end{align*}
At this stage, we claim the following result. The proof is postponed and will be presented in Appendix~\ref{Appendix B}.
\begin{lem}\label{log-type Caccioppoli lemma}
As $h\searrow 0$
\[
\log[u^{p-1}]_h \to \log u^{p-1} \quad \textrm{in}\,\,\,L^{\frac{p}{p-1}}(\Omega_{t_1,t_2}).
\]
\end{lem}

Thanks to this lemma and the preceding estimates, we infer that
\begin{align}\label{log-type Caccioppoli eq.1}
&\limsup_{h\searrow 0}\iint_{\Omega_T}\partial_t[|u|^{p-2}u]_h\chi_\varepsilon u^{1-p}\varphi^p\dxt\notag\\[3mm]
&\leq \limsup_{h\ \searrow 0}\left(\mathbf{I}_1+\mathbf{I}_2\right) \notag\\[3mm]
&\leq -\dashint_{t_1}^{t_1+\varepsilon} \int_\Omega \varphi^p\log u^{p-1}\dxt+\dashint_{t_2-\varepsilon}^{t_2} \int_\Omega \varphi^p\log u^{p-1} \dxt \notag\\[3mm]
&\quad \quad \quad \quad \quad -p\iint_{\Omega_{t_1,t_2}}\varphi^{p-1}\varphi^\prime \chi_\varepsilon \log u^{p-1} \dxt \notag\\[3mm]
&=-(p-1)\dashint_{t_1}^{t_1+\varepsilon} \int_\Omega \varphi^p\log u\dxt+(p-1)\dashint_{t_2-\varepsilon}^{t_2} \int_\Omega \varphi^p\log u \dxt \notag\\[3mm]
&\quad \quad \quad \quad \quad -p(p-1)\iint_{\Omega_{t_1,t_2}}\varphi^{p-1}\varphi^\prime \chi_\varepsilon \log u \dxt.
\end{align}
On the other hand, since by Lemma~\ref{mollification lemma}-(ii) $\left[\mathbf{A}(x,t,u,Du)\right]_h \to  \mathbf{A}(x,t,u,Du)$ in $L^{\frac{p}{p-1}}(\Omega_T)$ as $h \searrow 0$, we observe that
\begin{align*}
&\lim_{h\searrow 0}\iint_{\Omega_T}\left[\mathbf{A}(x,t,u,Du)\right]_h\cdot D\left(\chi_\varepsilon u^{1-p}\varphi^p\right)\dxt \\[3mm]
&=\iint_{\Omega_T}\mathbf{A}(x,t,u,Du)\cdot D\left(\chi_\varepsilon u^{1-p}\varphi^p\right)\dxt \\[3mm]
&=\iint_{\Omega_T}\chi_\varepsilon\mathbf{A}(x,t,u,Du)\cdot \left(-(p-1)u^{-p}Du\varphi^p+p\varphi^{p-1}D\varphi u^{1-p}\right)\dxt \\[3mm]
&\leq p\iint_{\Omega_T}\Big(\chi_\varepsilon^{\frac{1}{p}}\varphi|D (\log u)|\Big)^{p-1}\Big(c_1\chi_\varepsilon^{\frac{1}{p}}|D\varphi|\Big)\dxt  -c_0(p-1)\iint_{\Omega_T}\chi_\varepsilon\varphi^p|D(\log u)|^p\dxt,
\end{align*}
which together with Young's inequality with conjugate exponents $\left(\frac{p}{p-1}, p\right)$ now yields,
\begin{align}\label{log-type Caccioppoli eq.2}
&\lim_{h\searrow 0}\iint_{\Omega_T}\left[\mathbf{A}(s,t,u,Du)\right]_h\cdot D\left(\chi_\varepsilon u^{1-p}\varphi^p\right)\dxt \notag\\[3mm]
&\leq -\frac{p-1}{2}c_0\iint_{\Omega_{t_1,t_2}}\chi_\varepsilon |D(\log u)|^p\dxt+c(p,c_0,c_1)\iint_{\Omega_{t_1,t_2}}\chi_\varepsilon |D\varphi|^p\dxt.
\end{align}
As argued in Step 2 in Appendix~\ref{Appendix A}, we obtain
\begin{align}\label{log-type Caccioppoli eq.3}
&\lim_{h \searrow 0}\int_0^T\iint_{\bR^n \times \bR^n}\left[U(x,y,t)K(x,y,t)\right]_h\chi_\varepsilon(t) \left(u(x,t)^{1-p}\varphi(x,t)^p-u(y,t)^{1-p}\varphi(y,t)^p\right)\dxyt \notag\\[4mm]
&=\int_{t_1}^{t_2}\iint_{\bR^n \times \bR^n}U(x,y,t)K(x,y,t)\chi_\varepsilon(t) \left(u(x,t)^{1-p}\varphi(x,t)^p-u(y,t)^{1-p}\varphi(y,t)^p\right)\dxyt
\end{align}
and
\begin{equation}\label{log-type Caccioppoli eq.4}
\lim_{h\searrow 0}\int_\Omega|u|^{p-2}u(0)\left(\frac{1}{h}\int_0^Te^{-\frac{s}{h}}\chi_\varepsilon(s)u(x,s)^{1-p}\varphi(x,s)^p\ds\right)\dx=0.
\end{equation}
Therefore, collecting the preceding estimates~\eqref{log-type Caccioppoli eq.1}--\eqref{log-type Caccioppoli eq.4} and subsequently, passing to the limit  $\varepsilon\searrow 0$, we finally gain
\begin{align*}
&(p-1)\left[\int_\Omega \varphi^p\log u\dx\right]_{t=t_1}^{t_2}-p(p-1)\iint_{\Omega_{t_1,t_2}}\varphi^{p-1}\varphi^\prime \log u\dxt\\[3mm]
&- \frac{p-1}{2}c_0\iint_{\Omega_{t_1,t_2}}|D(\log u)|^p\dxt+c(p,c_0,c_1)\iint_{\Omega_{t_1,t_2}}|D\varphi|^p\dxt\\[3mm]
&+\int_{t_1}^{t_2}\iint_{\bR^n \times \bR^n}U(x,y,t)K(x,y,t) \left(u(x,t)^{1-p}\varphi(x,t)^p-u(y,t)^{1-p}\varphi(y,t)^p\right)\dxyt \geq 0,
\end{align*}
which in turn implies the desired inequality.
\end{proof}
We next deduce the log-type estimate for supersolutions.
\begin{prop}\label{log-type est.}
Assume that $u$ is a weak supersolution to~\eqref{maineq} satisfying $u \geq m>0$ in $\bR^n \times (t_0-\rho^p,t_0+\rho^p)$. Given $\sigma \in (0,1)$, let $\chi=\chi(x)\in C^\infty(\bR^n)$ be a smooth radial symmetric function such that
\[
\begin{cases}
0 \leq \chi \leq 1\quad \textrm{in}\,\,\,\bR^n,\quad \supp\,\chi\subset B_{\frac{1+\sigma}{2}\rho}(x_0); \\[2mm]
\chi \equiv 1\quad \textrm{on}\,\,\,B_{\sigma \rho}(x_0),\quad |D\chi|\leq \dfrac{c}{(1+\sigma)\rho}.
\end{cases}
\]
Set
\[
\beta:=\int_{B_\rho(x_0)}\chi^p(x)\log u(x,t_0)\dx.
\]
Then, there exist positive constant $\overline{C}=\overline{C}(n,s,p,\sigma)$ such that
\begin{align*}
\Big|Q_{\sigma \rho}^-(z_0) \cap \big\{\log u>\lambda+\beta\big\}\Big| &\leq \frac{\overline{C}}{\lambda^{p-1}}\left|Q_{\sigma \rho}^-(z_0)\right|
\intertext{and}
\Big|Q_{\sigma \rho}^+(z_0) \cap \big\{\log u<-\lambda+\beta\big\}\Big| &\leq \frac{\overline{C}}{\lambda^{p-1}}\left|Q_{\sigma \rho}^+(z_0)\right|
\end{align*}
\end{prop}
\medskip

Before proving Proposition~\ref{log-type est.}, we need the following technical lemma.
\begin{lem}\label{tech.}
With $p>1$ let $g(\tau):=\dfrac{1-\tau^{1-p}}{1-\tau}$ for  $\tau \in (0,1)$ and $A>0$. Then, 
\[
(1-\tau)^p\left(g(\tau)+A\right) \leq \left(A-(p-1)\right)\left[\log \left(\frac{1}{\tau}\right)\right]^p\quad \forall \tau \in (0,1) 
\]
holds true.
\end{lem}
\begin{proof}
Rearranging $g(\tau)$ as
\[
g(\tau)=-\frac{p-1}{1-\tau}\int_\tau^1\rho^{-p}\d\rho=-(p-1)\dashint_\tau^1\rho^{-p}\d\rho.
\]
By H\"{o}lder's inequality, we infer that
\[
\frac{1}{1-\tau}\log\left(\frac{1}{\tau}\right)=\dashint_\tau^1\rho^{-1}\d\rho \leq \left(\,\dashint_\tau^1\rho^{-p}\d\rho\right)^{\frac{1}{p}}
\]
and therefore, 
\begin{equation*}
(1-\tau)^pg(\tau) \leq -(p-1)\left[\log \left(\frac{1}{\tau}\right)\right]^p.
\end{equation*}
Since $1 \leq 1/\rho$ for any $\rho \in (\tau,1)$,
\begin{equation*}
(1-\tau)^p\leq \left(\int_\tau^1\frac{1}{\rho}\d\rho\right)^p=\left[\log \left(\frac{1}{\tau}\right)\right]^p.
\end{equation*}
Combining these displays, the result in turn follows.
\end{proof}
\begin{proof}[\normalfont\textbf{Proof of Proposition~\ref{log-type est.}}]
Following the argument in the literatures~\cite[Lemma 6.1]{Kinnunen-Kuusi} and~\cite[Lemma 1.3]{DiCKP2}, we prove the claim. 

\emph{Step 1}. In the first step, we derive the preliminarily result. For this, let us define
\[
V(t):=\frac{1}{N}\int_{B_\rho(x_0)}\left(\log u(x,t)-\beta\right)\chi(x)^p\dx,\quad N:=\int_{B_\rho(x_0)}\chi(x)^p\dx
\]
with $\displaystyle \beta=\int_{B_\rho(x_0)}\chi(x)^p\log u(x,t_0)\dx$. By definition, 
\[
V(t_0)=0,\quad \left|B_{\sigma \rho}\right|\leq N \leq \left|B_\rho\right|.
\]
Applying Lemma~\ref{log-type Caccioppoli} to $\varphi(x,t)=\chi(x)$ yields that
\begin{align*}
&\int_{t_1}^{t_2}\int_{B_\rho(x_0)}\chi^p\left|D(\log u)\right|^p\dxt-\left[\int_{B_\rho(x_0)} \chi^p \log u\dx\right]_{t=t_1}^{t_2} \\[4mm]
&\quad  \leq c\int_{t_1}^{t_2}\int_{B_\rho(x_0)}|D\chi|^p\dxt\\[4mm]
&\quad \quad +c\int_{t_1}^{t_2}\iint_{\bR^n \times \bR^n}U(x,y,t)K(x,y,t)\left(u(x,t)^{1-p}\chi(x,t)^p-u(y,t)^{1-p}\chi(y,t)^p\right)\dxyt
\end{align*}
whenever $t_0-(\sigma \rho)^p\leq t_1<t_2\leq t_0+(\sigma \rho)^p$.
In view of the weighted Poincar\'{e} inequality, it holds that
\begin{align*}
\int_{B_\rho(x_0)}\left|D\left(\log u-\beta \right)\right|^p\chi^p(x)\dx &\geq \frac{1}{c\rho^p}\int_{B_\rho(x_0)}\left|\log u-\beta-V(t)\right|^p\chi^p(x)\dx\\[4mm]
&\geq \frac{1}{c\rho^p}\int_{B_{\sigma\rho}(x_0)}\left|\log u-\beta-V(t)\right|^p\dx
\end{align*}
and therefore we get, for every $t_0-(\sigma \rho)^p\leq t_1<t_2\leq t_0+(\sigma \rho)^p$,
\begin{align}\label{log-type est. eq.1}
&\frac{1}{c\rho^pN}\int_{t_1}^{t_2}\int_{B_\rho(x_0)}\left|\log u-\beta-V(t)\right|^p\dx+\Big[V(t)\Big]_{t_1}^{t_2}\notag\\[4mm]
&\quad  \leq \frac{c}{N}\int_{t_1}^{t_2}\int_{B_\rho(x_0)}|D\chi|^p\dxt \notag\\[4mm]
&\quad \quad +\frac{c}{N}\int_{t_1}^{t_2}\iint_{\bR^n \times \bR^n}U(x,y,t)K(x,y,t)\left(u(x,t)^{1-p}\chi(x,t)^p-u(y,t)^{1-p}\chi(y,t)^p\right)\dxyt \notag \\[4mm]
&\quad=:\mathbf{I}+\mathbf{II},
\end{align}
where the definition of $\mathbf{I}$ and $\mathbf{II}$ are obvious from the context.

By definition, $\mathbf{I}$ is estimated as
\[
\mathbf{I} \leq \frac{c|B_\rho(x_0)|}{(1+\sigma)^p\rho^p}\frac{t_2-t_1}{N}.
\]
The fractional integral term $\mathbf{II}$ is split into two terms:
\begin{align}\label{log-type est. eq.2}
\mathbf{II}&=\frac{c}{N}\left(\int_{t_1}^{t_2}\iint_{B_\rho \times B_\rho}(\cdots)\dxyt+2\int_{t_1}^{t_2}\iint_{B_\rho \times \left(\bR^n \setminus B_\rho\right)}(\cdots)\dxyt\right) \notag\\[2mm]
&=:\frac{c}{N}\big(\mathbf{II}_1+2\mathbf{II}_2\big).
\end{align}
In order to estimate $\mathbf{II}_1$, we distinguish between the two cases $u(x,t)>u(y,t)$ and $u(x,t) \leq u(y,t)$. In the first case, applying Lemma~\ref{Gamma} with $a=\chi(x)$, $b=\chi(y)$ and
\[
\varepsilon=\delta\,\frac{u(x,t)-u(y,t)}{u(x,t)} \in (0,1)\quad \textrm{with}\quad \delta \in (0,1),
\]
the integrand of $\mathbf{II}_1$ is estimated as
\begin{align*}
&U(x,y,t)K(x,y,t)\left(\frac{\chi(x)^p}{u(x,t)^p}-\frac{\chi(y)^p}{u(y,t)^p}\right) \\[3mm]
&=K(x,y,t)\Big(u(x,t)-u(y,t)\Big)^{p-1}\left(\frac{\chi(x)^p}{u(x,t)^p}-\frac{\chi(y)^p}{u(y,t)^p}\right) \\[3mm]
&\leq K(x,y,t)\left(\frac{u(x,t)-u(y,t)}{u(x,t)}\right)^{p-1}\chi(y)^p\left[1+c\delta\,\frac{u(x,t)-u(y,t)}{u(x,t)} -\left(\frac{u(x,t)}{u(y,t)}\right)^{p-1}\right] \\[3mm]
&\quad \quad \quad \quad \quad \quad+c\delta^{1-p} K(x,y,t)\big|\chi(x)-\chi(y)\big|^p \\[3mm]
&=K(x,y,t)\left(\frac{u(x,t)-u(y,t)}{u(x,t)}\right)^{p}\left[g\left(\frac{u(y,t)}{u(x,t)}\right)+c\delta \right]+c\delta^{1-p} K(x,y,t)\big|\chi(x)-\chi(y)\big|^p,
\end{align*} 
where in the last line we let $g(\tau):=\dfrac{1-\tau^{1-p}}{1-\tau}$ for $\tau \in (0,1)$. Therefore, Lemma~\ref{tech.} applied with $\tau\equiv \dfrac{u(y,t)}{u(x,t)}$ and $A\equiv c\delta$ and choosing $\delta\equiv \dfrac{p-1}{2^pc}$ implies that
\begin{align*}
&U(x,y,t)K(x,y,t)\left(\frac{\chi(x)^p}{u(x,t)^p}-\frac{\chi(y)^p}{u(y,t)^p}\right) \\[3mm] 
&\leq -\frac{(p-1)(2^p-1)}{2^p}K(x,y,t)\left[\log\left(\frac{u(x,t)}{u(y,t)}\right)\right]^p\chi(y)^p+c\left(\frac{p-1}{2^pc}\right)^{1-p}K(x,y,t)\big|\chi(x)-\chi(y)\big|^p.
\end{align*}
In the case $u(x,t)=u(y,t)$, this estimate trivially holds and, by the interchanging the role of $u(x,t)$ and $u(y,t)$, we also recover this estimate in the case $u(x,t)<u(y,t)$.
Therefore, we finally deduce that
\begin{align}\label{log-type est. eq.3}
\mathbf{II}_1&\leq -c\int_{t_1}^{t_2}\iint_{B_\rho\times B_\rho}K(x,y,t)\left|\log\left(\frac{u(x,t)}{u(y,t)}\right)\right|^p\chi(y)^p\dxyt \notag\\[4mm]
&\quad \quad +c\int_{t_1}^{t_2}\iint_{B_\rho\times B_\rho}K(x,y,t)\big|\chi(x)-\chi(y)\big|^p\dxyt \notag\\[4mm]
&\leq -c\int_{t_1}^{t_2}\iint_{B_\rho\times B_\rho}K(x,y,t)\left|\log\left(\frac{u(x,t)}{u(y,t)}\right)\right|^p\chi(y)^p\dxyt \notag\\[4mm]
&\quad \quad +\frac{c\Lambda}{(1+\sigma)^p\rho^p}(t_2-t_1)|B_\rho|\sup_{x \in B_\rho}\left(\int_{B_\rho}\frac{\dy}{|x-y|^{n-(1-s)p}}\right) \notag\\[4mm]
&\leq \frac{c}{p(1-s)}\frac{\rho^{n-sp}}{(1+\sigma)^p}(t_2-t_1)\end{align}
for a constant $c=c(n,s,p,\Lambda)$, where we again used that 
\[
 \int_{B_\rho}\frac{\dy}{|x-y|^{n-(1-s)p}}=|B_\rho|\frac{\rho^{(1-s)p}}{(1-s)p}.
 \]

Next, recalling the fact that $\supp\,\chi \in B_{\frac{1+\sigma}{2}}(x_0)$ and
\[
\frac{(u(x,t)-u(y,t))_+^{p-1}}{u(x,t)^{p-1}}\leq 1\qquad \forall (x,y,t) \in B_{\frac{1+\sigma}{2}}(x_0)\times (\bR^n\setminus B_\rho(x_0))\times (t_1,t_2),
\]
we infer that
\begin{align*}
\mathbf{II}_2&=2\iint_{B_{\frac{1+\sigma}{2}\rho}(x_0)\times (\bR^n\setminus B_\rho(x_0))}U(x,y,t)K(x,y,t)u(x,t)^{1-p}\chi(x)^p\dxyt \\[4mm]
&\leq 2\iint_{B_{\frac{1+\sigma}{2}\rho}(x_0)\times (\bR^n\setminus B_\rho(x_0))}(u(x,t)-u(y,t))_+^{p-1}K(x,y,t)u(x,t)^{1-p}\dxyt \\[4mm]
&\leq 2\Lambda|B_{\frac{1+\sigma}{2}\rho}(x_0)|(t_2-t_1)\sup_{x\in B_{\frac{1+\sigma}{2}\rho}(x_0)}\int_{\bR^n \setminus B_\rho(x_0)}\frac{\dy}{|x-y|^{n+sp}}.
\end{align*}
Similarly as before, since
\[
\frac{|y-x_0|}{|y-x|}\leq \frac{|y-x|+|x-x_0|}{|y-x|}=1+\frac{|x-x_0|}{|y-x|}
\]
and
\[
|y-x| \geq |y-x_0|-|x_0-x| \geq \rho-\frac{1+\sigma}{2}\rho=\frac{1-\sigma}{2}\rho
\]
holds again for any $x \in B_{\frac{1+\sigma}{2}}(x_0)$ and $y \in \bR^n \setminus B_\rho(x_0)$, we get 
\[
\dfrac{|y-x_0|}{|y-x|} \leq \dfrac{2}{1-\sigma}\qquad \forall x \in B_{\frac{1+\sigma}{2}\rho}(x_0),\quad \forall y \in \bR^n \setminus B_\rho(x_0)
\]
and therefore,
\begin{align}\label{log-type est. eq.4}
\mathbf{II}_2 &\leq c\frac{|B_{\frac{1+\sigma}{2}}(x_0)|}{(1-\sigma)^{n+sp}}(t_2-t_1)\sup_{x\in B_{\frac{1+\sigma}{2}\rho}(x_0)}\int_{\bR^n \setminus B_\rho(x_0)}\frac{\dy}{|x-x_0|^{n+sp}} \notag\\[4mm]
&\leq c\frac{(1+\sigma)^n}{(1-\sigma)^{n+sp}}\rho^{n-sp}(t_2-t_1)
\end{align}
with a constant $c\equiv c(n,s,p,\Lambda)$, where $\displaystyle \int_{\bR^n \setminus B_\rho(x_0)}\frac{\dy}{|y-x_0|^{n+sp}}=\frac{c(n)}{sp}\rho^{-sp}$ is used again. Merging estimates~\eqref{log-type est. eq.3}--\eqref{log-type est. eq.4} with~\eqref{log-type est. eq.2}, we obtain
\begin{align*}
\mathbf{II} \leq \frac{c}{N}\left(\frac{1}{(1+\sigma)^p}+\frac{(1+\sigma)^n}{(1-\sigma)^{n+sp}}\right)\rho^{n-sp}(t_2-t_1).
\end{align*}
Combining the preceding estimates with~\eqref{log-type est. eq.1} implies that
\begin{align*}
&\frac{1}{c\rho^pN}\,\dashint_{t_1}^{t_2}\int_{B_\rho(x_0)}\left|\log u-\beta-V(t)\right|^p\dxt+\frac{V(t_2)-V(t_1)}{t_2-t_1} \\[3mm]
&\leq \frac{c|B_\rho|}{N(1+\sigma)^p\rho^p}+\frac{c}{N}\left(\frac{1}{(1+\sigma)^p}+\frac{(1+\sigma)^n}{(1-\sigma)^{n+sp}}\right)\rho^{n-sp}, 
\end{align*}
which together with the fact that $|B_\rho| \geq N \geq \sigma^n |B_\rho|$ yields 
\[
\frac{1}{c\rho^p|B_\rho|}\,\dashint_{t_1}^{t_2}\int_{B_{\sigma\rho}(x_0)}\left|\log u-\beta-V(t)\right|^p\dxt+\frac{V(t_2)-V(t_1)}{t_2-t_1} \leq C(\rho^{-p}+\rho^{-sp})
\]
with a constant $C=C(n,s,p,\Lambda,\sigma)$. Denoting
\begin{align*}
w(x,t)&:=\log u(x,t)-\beta-C(\rho^{-p}+\rho^{-sp})(t-t_0), \\[2mm]
W(t)&:=V(t)-C(\rho^{-p}+\rho^{-sp})(t-t_0)
\end{align*}
with noticing that $W(t_0)=0$, the above display becomes
\[
\frac{1}{c\rho^p|B_\rho|}\,\dashint_{t_1}^{t_2}\int_{B_{\sigma\rho}(x_0)}\left|w-W(t)\right|^p\dxt+\frac{W(t_2)-W(t_1)}{t_2-t_1} \leq 0.
\]
Since by the monotone increasing of the function $t \mapsto W(t)$, $W(t)$ is differentiable for almost everywhere $t \in (t_0-\rho^p,t_0+\rho^p)$, letting $t_2 \searrow t_1=t$ in the above formula gives that
\begin{equation}\label{log-type. est. eq.5}
\frac{1}{c\rho^p|B_\rho|}\int_{B_{\sigma\rho}(x_0)}\left|w-W(t)\right|^p\dx+W^\prime(t) \leq 0
\end{equation}
for almost everywhere $t \in (t_0-\rho^p,t_0+\rho^p)$.
\medskip

\emph{Step 2}. In this step we shall estimate the Lebesgue measure of
\begin{align*}
\Sigma_\lambda^- (t)&:=\Big\{x \in B_{\sigma\rho}(x_0): w(x,t) >\lambda \Big\}\quad \textrm{for}\quad t\in(t_0-(\sigma \rho)^p,\rho^p)
%&=\Big\{x \in B_{\sigma\rho}(x_0): \log u(x,t) >\beta+C(\rho^{-p}+\rho^{sp})(t-t_0)\Big\}
\shortintertext{and}
\Sigma_\lambda^+ (t)&:=\Big\{x \in B_{\sigma\rho}(x_0): w(x,t) <-\lambda \Big\}\quad \textrm{for}\quad t\in(t_0,t_0+(\sigma \rho)^p),
\end{align*}
where $\lambda>0$ is arbitrarily given. Since $W(t) < W(t_0)=0$ for $t \in (t_0-(\sigma \rho)^p,t_0)$ it holds that
\begin{align*}
\int_{B_{\sigma \rho}(x_0)}|w-W(t)|^p\dx \geq \int_{\Sigma_\lambda^- (t)}|\underbrace{\lambda-W(t)}_{>0}|^p\dx =(\lambda-W(t))^p\left|\Sigma_\lambda^- (t)\right|,
\end{align*}
and therefore, by~\eqref{log-type. est. eq.5}, for almost everywhere $t \in (t_0-(\sigma \rho)^p,t_0)$, we have
\[
\frac{\left|\Sigma_\lambda^- (t)\right|}{c\rho^{p}|B_\rho|}+\frac{W^\prime(t)}{(\lambda-W(t))^p}\leq 0.
\]
Integrating with respect to $t \in (t_0-(\sigma \rho)^p,t_0)$ yields
\[
\int_{t_0-(\sigma\rho)^p}^{t_0}\left|\Sigma_\lambda^- (t)\right|\dt \leq \frac{\overline{C}}{\lambda^{p-1}}\left|Q_{\sigma \rho}^-(z_0)\right|
\]
for a  constant $\overline{C}=\overline{C}(n,s,p,\sigma)$. Therefore, since $t-t_0<0$ for $t \in \left(t_0-(\sigma \rho)^p, t_0\right)$ we conclude that
\[
\Big|Q^-_{\sigma \rho}(z_0) \cap \left\{\log u>\lambda+\beta\right\}\Big| \leq \int_{t_0-(\sigma\rho)^p}^{t_0}\left|\Sigma_\lambda^- (t)\right|\dt \leq \frac{\overline{C}}{\lambda^{p-1}}\left|Q_{\sigma \rho}^-(z_0)\right|,
\]
as desired. 
By the same reasoning, we can deduce that
\[
\int_{t_0}^{t_0+(\sigma\rho)^p}\left|\Sigma_\lambda^+ (t)\right|\dt \leq \frac{\overline{C}}{\lambda^{p-1}}\left|Q_{\sigma \rho}^+(z_0)\right|
\]
and hence
\[
\Big|Q^+_{\sigma \rho}(z_0) \cap \left\{\log u<-\lambda+\beta\right\}\Big| \leq \int_{t_0}^{t_0+(\sigma\rho)^p}\left|\Sigma_\lambda^+ (t)\right|\dt \leq \frac{\overline{C}}{\lambda^{p-1}}\left|Q_{\sigma \rho}^+(z_0)\right|.
\]
This completes the proof of the lemma.
\end{proof}

%%%%%%%%%%%%%%%%%%%%%%%
%%%%%%%%%%%%%%%%%%%%%%%
We conclude this section by deducing the following helpful lemma that will be used later in the proof of Theorem~\ref{mainthm}.
\begin{lem}\label{supsub}
Let $p >1$ and $s\in (0,1)$. If $u$ is a weak supersolution to~\eqref{maineq} satisfying $u \geq m >0$ in $\bR^n \times (0,T)$, then $v=u^{-1}$ is a positive subsolution to~\eqref{maineq} replaced $\mathbf{A}$ by $\widetilde{\mathbf{A}}$, where the vactor field $\widetilde{\mathbf{A}}:\Omega_T \times \bR \times\bR^n \to \bR^n$ is defined by
\[
\widetilde{\mathbf{A}}\left(x,t,y,\zeta\right):=-\bar{y}^{2(p-1)}\mathbf{A}\left(x,t,\bar{y}^{-1},-\bar{y}^{-2}\zeta\right) \quad{\textrm{with}}\quad \bar{y}:=\min\left\{y,\frac{1}{m}\right\}
\]
and possesses the same structure constants as $\mathbf{A}$. 
\end{lem}
\begin{proof}
First of all, by definition we see that $\widetilde{\mathbf{A}}$ is measurable with respect to $(x,t) \in \Omega_T$ for every $(y,\zeta)  \in \bR \times \bR^n$ and continuous with respect to $(y,\zeta)$ for almost everywhere $(x,t) \in \Omega_T$. A straightforward computation shows that $\widetilde{\mathbf{A}}$ fulfills the structure condition
\[
\begin{cases}
\widetilde{\mathbf{A}}(x,t,y,\zeta) \cdot \zeta \geq c_0 |\zeta|^p \\[2mm]
\left|\widetilde{\mathbf{A}}(x,t,y,\zeta)\right|\leq c_1|\zeta|^{p-1}
\end{cases}
\]
and therefore, the vector field $\widetilde{\mathbf{A}}$ is Carath\'{e}odory map fulfilling the same structural condition as $\mathbf{A}$.
\medskip

Let $t_1 \in (0,T)$ be an arbitrary time. For any $\delta>0$ small, let us define the cutoff function $\psi_\delta(t)$ with respect to time as
\[
\psi_\delta(t):=\begin{cases}
0,\quad &\,\,\,t \in (0, t_1), \\
\frac{1}{\delta}(t-t_1), \quad &\,\,\,t \in [t_1, t_1+\delta),\\
1,\quad &\,\,\,t \in [t_1+\delta, T).
\end{cases}
\]
Now, in the weak formulation~\eqref{D2'}, testing $\varphi=\psi_\delta u^{2(1-p)}\phi$ for any nonnegative $\phi \in \mathcal{T}$, we have
\begin{align}\label{supsub eq.1}
&\iint_{\Omega_T}\left(\partial_t[u^{p-1}]_{h}\,\psi_\delta u^{2(1-p)}\phi+\left[\mathbf{A}\left(x,t,u,Du\right)\right]_h\cdot D\left(\psi_\delta u^{2(1-p)}\phi\right)\right)\dxt \notag \\[3mm]
&\quad +\int_0^T\iint_{\bR^n \times \bR^n} \left[U(x,y,t)K(x,y,t)\right]_h \psi_\delta(t) \left(u(x,t)^{2(1-p)}\phi(x,t)-u(y,t)^{2(1-p)}\phi(y,t)\right)\dxyt  \notag\\[3mm]
&\geq
\int_\Omega u(0)^{p-1}\left(\dfrac{1}{h}\int_0^Te^{\frac{s}{h}}\psi_\delta(s)u^{2(1-p)}\phi(x,s)\ds\right)\dx.
\end{align}
For the evolutionary term, 
\begin{align*}
\iint_{\Omega_T}\partial_t[u^{p-1}]_{h}\,u^{2(1-p)}\psi_\delta \phi\dxt&=\iint_{\Omega_{t_1,T}}\partial_t[u^{p-1}]_{h}\left(u^{2(1-p)}-[u^{p-1}]_h^{-2}\right)\psi_\delta\phi\dxt \\[3mm]
&\quad +\iint_{\Omega_{t_1,T}}\partial_t[u^{p-1}]_{h}[u^{p-1}]_h^{-2}\psi_\delta \phi\dxt \\[3mm]
&=:\mathbf{I}_h+\mathbf{II}_h.
\end{align*}
Lemma~\ref{mollification lemma}-(ii) and inequality~\eqref{algs'} with $\alpha=3$  in Lemma~~\ref{Algs'} imply that
\[
\mathbf{I}_h=\iint_{\Omega_{t_1,T}}\psi_\delta \phi \frac{u^{p-1}-[u^{p-1}]_h}{h}\cdot \frac{[u^{p-1}]_h^2-u^{2(p-1)}}{[u^{p-1}]_h^2u^{2(p-1)}}\dxt \leq 0.
\]
For the observation of $\mathbf{II}_h$, we require the following lemma, whose proof will be given in Appendix~\ref{Appendix B}.  
\begin{lem}\label{supsub lemma}
Suppose that $u \geq m >0$ in $\bR^n\times (0,T)$ and fix $t_1 \in (0,T)$. Then
\[
[u^{p-1}]^{-1}_h \to u^{1-p}=v^{p-1} \quad \textrm{in}\,\,\,L^{\frac{p}{p-1}}(\Omega_{t_1,T}) \quad \textrm{as}\,\,\,h \searrow 0.
\]
\end{lem}
By Lemma~\ref{supsub lemma}, we have
\begin{align*}
\lim_{h \searrow 0}\,\mathbf{II}_h&=\lim_{h \searrow 0}\iint_{\Omega_{t_1,T}}\psi_\delta \phi\partial_t \left(-[u^{p-1}]_h^{-1}\right)\dxt \\[3mm]
&=\lim_{h \searrow 0}\iint_{\Omega_{t_1,T}}[u^{p-1}]_h^{-1}\partial_t(\psi_\delta\phi)\dxt \\[3mm]
&=\iint_{\Omega_{t_1,T}}v^{p-1}\partial_t(\psi_\delta\phi)\dxt \\[3mm]
&=\dashint_{t_1}^{t_1+\delta}\int_\Omega v^{p-1}\phi\dxt+\iint_{\Omega_{t_1,T}}v^{p-1}\psi_\delta\partial_t \phi\dxt.
\end{align*}
Combining the preceding estimates above, we infer that
\begin{align}\label{supsub eq.2}
&\limsup_{h \searrow 0}\iint_{\Omega_T}\partial_t[u^{p-1}]_{h}\,u^{2(1-p)}\psi_\delta\phi\dxt  \notag\\[3mm]
&\leq\limsup_{h \searrow 0} \left(\mathbf{I}_h+\mathbf{II}_h\right)
\leq \dashint_{t_1}^{t_1+\delta}\int_\Omega v^{p-1}\phi\dxt+\iint_{\Omega_{t_1,T}}v^{p-1}\psi_\delta\partial_t \phi\dxt
\end{align}
We shall estimate the spatial term involving $\mathbf{A}$.
As mentioned before, since
\[
[\mathbf{A}(x,t,u,Du)]_h \to \mathbf{A}(x,t,u,Du)\quad \textrm{in}\quad L^{\frac{p}{p-1}}(\Omega_T), 
\]
one can easily check that
\begin{align*}
&\lim_{h \searrow 0}\iint_{\Omega_T}\left[\mathbf{A}\left(x,t,u,Du\right)\right]_h \cdot D\left(u^{2(1-p)}\psi_\delta\phi\right)\dxt \\[3mm]
&=\iint_{\Omega_T}\mathbf{A}\left(x,t,u,Du\right) \cdot D\left(u^{2(1-p)}\psi_\delta \phi\right)\dxt \\[3mm]
&=\iint_{\Omega_{t_1,T}}\psi_\delta v^{2(p-1)}\mathbf{A}\left(x,t,v^{-1},-v^{-2}Dv\right)\cdot D\phi\dxt \\[3mm]
&\quad\quad \quad  +2(p-1)\iint_{\Omega_{t_1,T}}\psi_\delta\phi \,v^{2p-1}\mathbf{A}\left(x,t,v^{-1},-v^{-2}Dv\right)\cdot D\phi\dxt \\[3mm]
&=:\mathbf{I}+\mathbf{II},
\end{align*}
where the definition of $\mathbf{I}$ and $\mathbf{II}$ are clear form the context. 

By definition, $\mathrm{I}$ is rewritten as
\[
\mathbf{I}=-\iint_{\Omega_{t_1,T}}\psi_\delta\widetilde{\mathbf{A}}\left(x,t,v,Dv\right)\cdot D\phi \dxt. 
\]
For term $\mathbf{II}$ we have
\begin{align*}
\mathbf{II}&=-2(p-1)\iint_{\Omega_{t_1,T}}\psi_\delta\phi \,v^{-1}\widetilde{\mathbf{A}}\left(x,t,v,Dv\right)\cdot Dv\dxt \\[3mm]
&\leq -2(p-1)\iint_{\Omega_T}\psi_\delta\phi \,v^{-1}c_0|Dv|^p\dxt \leq 0.
\end{align*}
Combining the preceding estimates gives
\begin{align}\label{supsub eq.3}
&\lim_{h \searrow 0}\iint_{\Omega_T}\left[\mathbf{A}\left(x,t,u,Du\right)\right]_h\cdot D\left(u^{2(1-p)}\psi_\delta \phi\right)\dxt \notag\\[3mm]
&\leq \mathbf{I}+\mathbf{II} \leq -\iint_{\Omega_{t_1,T}}\psi_\delta \widetilde{\mathbf{A}}\left(x,t,v,Dv\right)\cdot D\phi\dxt.
\end{align}
For the fractional term in~\eqref{supsub eq.1}, using the same argument as Step 2 in Appendix~\ref{Appendix A}, we obtain that
\begin{align*}
&\lim_{h \searrow 0}\int_0^T\iint_{\bR^n \times \bR^n} \left[U(x,y,t)K(x,y,t)\right]_h\psi_\delta(t)\left(u(x,t)^{2(1-p)}\phi(x,t)-u(y,t)^{2(1-p)}\phi(y,t)\right)\dxyt \\[3mm]
&=\int_{t_1}^T\iint_{\bR^n \times \bR^n} U(x,y,t)K(x,y,t)\psi_\delta(t)\left(u(x,t)^{2(1-p)}\phi(x,t)-u(y,t)^{2(1-p)}\phi(y,t)\right)\dxyt \\[3mm]
&=:\mathbf{III}
\end{align*}
In order to estimate $\mathbf{III}$, we now follow the argument considered in~\cite{BGK}. Since
\[
U(x,y,t)=u(x,t)^{p-1}u(y,t)^{p-1}\left|v(x,t)-v(y,t)\right|^{p-2}\left(v(x,t)-v(y,t)\right),
\]
the integrand of $\mathbf{III}$ is written as
\begin{align*}
&U(x,y,t)K(x,y,t)\left(u(x,t)^{2(1-p)}\phi(x,t)-u(y,t)^{2(1-p)}\phi(y,t)\right) \\[3mm]
&=-V(x,y,t)K(x,y,t)\left[\left(\frac{v(x,t)}{v(y,t)}\right)^{p-1}\phi(x,t)-\left(\frac{v(y,t)}{v(x,t)}\right)^{p-1}\phi(y,t)\right]
\end{align*}
with the shorthand notation
\[
V(x,y,t):=\left|v(x,t)-v(y,t)\right|^{p-2}\left(v(x,t)-v(y,t)\right).
\]
We now distinguish between two cases $v(x,t) \geq v(y,t)$ and $v(x,t) <v(y,t)$. As we are considering the case $v(x,t) \geq v(y,t)$, since 
\[
\frac{v(x,t)}{v(y,t)}\geq 1,\quad \frac{v(y,t)}{v(x,t)} \leq 1
\]
we have the bound
\begin{align*}
&V(x,y,t)K(x,y,t)\left[\left(\frac{v(x,t)}{v(y,t)}\right)^{p-1}\phi(x,t)-\left(\frac{v(y,t)}{v(x,t)}\right)^{p-1}\phi(y,t)\right] \\[3mm]
&\geq V(x,y,t)K(x,y,t)\left(\phi(x,t)-\phi(y,t)\right).
\end{align*}
When $v(x,t) < v(y,t)$, since 
\[
\frac{v(x,t)}{v(y,t)}<1,\quad \frac{v(y,t)}{v(x,t)} >1
\]
we get
\begin{align*}
&V(x,y,t)\left[\left(\frac{v(x,t)}{v(y,t)}\right)^{p-1}\phi(x,t)-\left(\frac{v(y,t)}{v(x,t)}\right)^{p-1}\phi(y,t)\right] \\[3mm]
&\geq \left|v(x,t)-v(y,t)\right|^{p-2}\left(v(y,t)-v(x,t)\right)\left(\phi(y,t)-\phi(x,t)\right)\\[3mm]
&=V(x,y,t)\left(\phi(x,t)-\phi(y,t)\right)
\end{align*}
as well. Consequently, we conclude that
\begin{equation}\label{supsub eq.4}
\mathbf{III}\leq -\int_{t_1}^T\iint_{\bR^n \times \bR^n}V(x,y,t)K(x,y,t)\psi_\delta(t)\left(\phi(y,t)-\phi(x,t)\right)\dxyt.
\end{equation}
Since by $u \geq m >0$ in $\bR^n \times (0,T)$ and $\psi_\delta(s) \leq 1$ it holds that
\begin{equation}\label{supsub eq.5}
\lim_{h\searrow 0}\int_\Omega u(0)^{p-1}\left(\dfrac{1}{h}\int_0^Te^{\frac{s}{h}}u^{2(1-p)}\psi_\delta(s)\phi(x,s)\ds\right)\dx=0,
\end{equation}
collecting these estimations~\eqref{supsub eq.2}--\eqref{supsub eq.5} and, in~\eqref{supsub eq.1}, passing to the limit $h \searrow 0$ and subsequently, sending $\delta\ \searrow 0$, we conclude that 
\begin{align}\label{supsub eq.6}
&\iint_{\Omega_{t_1,T}}\left(v^{p-1}\partial_t\phi+\widetilde{\mathbf{A}}\left(x,t,v,Dv\right)\cdot D\phi\right)\dxt -\int_{\Omega} v(x,t_1)^{p-1}\phi(x,t_1)\dx \notag\\[3mm]
&\quad \quad \quad +\int_{t_1}^T\iint_{\bR^n \times \bR^n}V(x,y,t)K(x,y,t)\left(\phi(x,t)-\phi(y,t)\right)\dxyt \leq 0.
\end{align}
for every nonnegative  $\phi \in \mathcal{T}$. Now, we claim that
\begin{equation}\label{supsub eq.7}
\lim_{t_1 \searrow 0}\int_\Omega \phi(x,t_1)\dx=0.
\end{equation}
Indeed, for any nonnegative $\phi \in \mathcal{T}$, set
\[
\langle \phi \rangle_h:=
\left([\phi^{\frac{p}{2}}]_h\right)^{\frac{2}{p}}.
\]
By H\"{o}lder's inequality, we estimate 
\begin{align*}
0\leq \int_\Omega \phi(x,t_1)\dx&=\int_\Omega \left(\phi(x,t_1)-\langle \phi \rangle_h(x,t_1)\right)\dx+\int_\Omega \langle \phi \rangle_h(x,t_1)\dx \\[3mm]
&\leq \Big(\|\phi(t_1)-\langle \phi \rangle_h(t_1) \|_{L^p(\Omega)}+\|\langle \phi \rangle_h(t_1) \|_{L^p(\Omega)}\Big)|\Omega|^{\frac{p-1}{p}}.
\end{align*}
Since by Lemma~\ref{mollification lemma2}, for every $\varepsilon>0$ there exists $\delta>0$ such that
\[
\|\phi(t_1)-\langle \phi \rangle_h(t_1) \|_{L^p(\Omega)}<\varepsilon
\]
whenever $h \in (0,\delta)$ and $t_1 \in [0,T]$. This implies that
\begin{equation}\label{supsub eq.8}
0\leq \int_\Omega \phi(x,t_1)\dx \leq \Big(\varepsilon+\|\langle \phi \rangle_h(t_1) \|_{L^p(\Omega)}\Big)|\Omega|^{\frac{p-1}{p}}
\end{equation}
and moreover, in view of Lemma~\ref{mollification lemma2} and $\phi(x,0)=0$,
\[
\|\langle \phi \rangle_h(t_1) \|_{L^p(\Omega)} \to \|\langle \phi \rangle_h(0) \|_{L^p(\Omega)}=0 \quad \textrm{as}\,\,\,t_1 \searrow 0
\]
holds true. Passing to the limit $t_1\searrow 0$ in~\eqref{supsub eq.8} gives
\[
0\leq \lim_{t_1\searrow 0}\int_\Omega \phi(x,t_1)\dx \leq \varepsilon|\Omega|^{\frac{p-1}{p}},
\]
and subsequently, letting $\varepsilon\searrow 0$ proves ~\eqref{supsub eq.7}. 
\medskip

Thus, using~\eqref{supsub eq.7} and the assumption that $v^{-1}=u \geq m >0$ in $\bR^n \times (0,T)$, we have
\[
0 \leq\lim_{t_1\searrow 0}   \int_\Omega v(x,t_1)^{p-1}\phi(x,t_1)\dx \leq m^{1-p} \lim_{t_1\searrow 0}\int_\Omega \phi(x,t_1)\dx=0
\]
and therefore, using this and letting $t_1\searrow 0$ in~\eqref{supsub eq.6} conclude that
\begin{align*}
&\iint_{\Omega_{T}}\left(v^{p-1}\partial_t\phi+\widetilde{\mathbf{A}}\left(x,t,v,Dv\right)\cdot D\phi\right)\dxt \\[3mm]
&\quad \quad \quad +\int_{0}^T\iint_{\bR^n \times \bR^n}V(x,y,t)K(x,y,t)\left(\phi(x,t)-\phi(y,t)\right)\dxyt \leq 0,
\end{align*}
proving the claim.
\end{proof}
%%%%%%%%%%%%%%%%%%%%%%%
%%%%%%%%%%%%%%%%%%%%%%%

%%%%%%%%%%%%%%%%%%%%%%%%%
\section{Quantitative estimates for subsolutions}\label{Sect. 4}

In this section, we give quantitative estimates for subsolutions. We begin by deriving the reverse H\"{o}lder inequality for supersolutions, then we prove the local boundedness of supersolutions under the assumption that $u \geq m>0$ in $\bR^n \times (t_0,t_0+\rho^p)$.
By the same reasoning as Lemma~\ref{Caccioppoli1}, the following lemma holds true:
\begin{lem}[Caccioppoli type estimate for subsolutions]\label{Caccioppoli2}
Let $u$ be a weak subsolution to~\eqref{maineq} fulfilling $u \geq m>0$ in $\bR^n \times (t_0,t_0+\rho^p)$. With $p >1$ and $s \in (0,1)$ let $\varepsilon \in (0,p-1)$ and set $\alpha:=p-1+\varepsilon$. Then
\begin{align*}
&\sup_{t \in (t_0,t_0+\rho^p)}\int_{B_\rho(x_0) \times \{t\}}u^\alpha\varphi^p\dx+\iint_{Q^+_\rho(z_0)}|Du|^pu^{-\varepsilon-1}\varphi^p\dxt\\[4mm]
&\quad \quad \quad \quad \quad +\int_{t_0}^{t_0+\rho^p}\iint_{B_\rho(x_0)\times B_\rho(x_0)}\dfrac{\big|u(x,t)^{\frac{\alpha}{p}}\varphi(x,t)-u(y,t)^{\frac{\alpha}{p}}\varphi(y,t)\big|^p}{|x-y|^{n+sp}}\dxyt\\[4mm]
&\leq c\iint_{Q^+_\rho(z_0)}u^\alpha \varphi^{p-1}|\varphi_t|\dxt+c\iint_{Q^+_\rho(z_0)}u^\alpha |D\varphi|^p\dxt\\[4mm]
&\quad \quad +c\int_{t_0}^{t_0+\rho^p}\iint_{B_\rho(x_0)\times B_\rho(x_0)}\dfrac{\big(u(x,t)^{\alpha}+u(y,t)^{\alpha}\big)\big|\varphi(x,t)-\varphi(y,t)\big|^p}{|x-y|^{n+sp}}\dxyt \\[4mm]
&\quad \quad +c \left(\sup_{x \in \supp \,\varphi(\cdot, t)}\int_{\bR^n \setminus B_\rho(x_0)}\dfrac{\dy}{|x-y|^{n+sp}}\right)\iint_{Q^+_\rho(z_0)}u^\alpha\varphi^p\dxt
\end{align*}
holds whenever nonnegative $\varphi \in C^\infty_0(Q^+_\rho(z_0))$, where the constant $c \equiv c(p,c_0,c_1,\Lambda,\varepsilon)$. Note that the constant $c$ has a singularity near $\varepsilon=0$.
\end{lem} 
Following the same reasoning of Lemma~\ref{reverse lemma}, we obtain the Reverse H\"{o}lder type lemma for subsolutions.
\begin{lem}\label{reverse lemma'}
Suppose that  a weak subsolution $u$ to~\eqref{maineq} satisfies $u \geq m>0$ in $\bR^n \times (t_0,t_0+\rho^p)$. For any $\varepsilon>0$ let $\alpha:=p-1+\varepsilon$ and $\kappa:=\frac{n+p}{n}$. Then for any concentric cylinders $Q^+_r(z_0) \subset Q^+_\rho(z_0) \Subset \Omega_T$ the quantitative estimate 
\begin{align*}
\left(\fiint_{Q^+_r(z_0)}u^{\alpha \kappa}\dxt\right)^{\frac{1}{\alpha \kappa}}\leq c^{\frac{1}{\alpha \kappa}}\left[\left(\frac{\rho}{r}\right)^n\left(\frac{\rho}{\rho-r}\right)^{n+sp}(1+\rho^{(1-s)p})\right]^{\frac{1}{\alpha}}\left(\fiint_{Q^+_\rho(z_0)}u^\alpha\dxt\right)^{\frac{1}{\alpha}}
\end{align*}
holds true, where $c=c(n,s,p,c_0,c_1,\Lambda,\varepsilon)$ blows up as $\varepsilon \searrow  0$.
\end{lem}
\begin{prop}[Local boundedness for subsolutions]\label{local bounds}
Suppose that a weak subsolution $u$ to~\eqref{maineq} satisfies that $u \geq m>0$ in $\bR^n \times (t_0,t_0+\rho^p)$. Let $\sigma$ and $\tau$ fulfill $0<\sigma<\tau \leq 1$ and let $\gamma>0$ be an arbitrary number. Then, there exists a positive constant $c=c(n,s,p,c_0,c_1,\Lambda,\gamma)$ such that
\[
\sup_{Q^+_{\sigma\rho}(z_0)}u \leq c\left(\sigma^{-(1-s)p}\frac{1+\rho^{(1-s)p}}{(\tau-\sigma)^{n+sp}}\right)^{\frac{n+p}{p\gamma}}\left(\fiint_{Q^+_{\tau \rho}}u^{\gamma}\dxt\right)^{\frac{1}{\gamma}},
\]
holds whenever  concentric cylinders $Q^+_{\sigma \rho}(z_0) \subset Q^+_{\tau\rho}(z_0) \Subset \Omega_T$.
\end{prop}
\begin{proof}
 Let $q_0>p-1$. Fix $0<\sigma\leq \theta_1<\theta_2 \leq \tau\leq 1$ and define, for $i=0,1,2\ldots$,
\[
\rho_i:=\theta_1\rho+\frac{(\theta_2-\theta_1)\rho}{2^i}\quad ; \quad Q_i:= Q^+_{\rho_i}(z_0)=B_{\rho_i}(x_0)\times (t_0,t_0+\rho_i^p)
\]
and therefore the following inclusions hold inductively:
\[
\rho_0=\theta_2\rho \geq \cdots \geq \rho_i \geq \rho_\infty=\theta_1 \rho,\qquad Q_0=Q^+_{\theta_2\rho} \supset \cdots \supset Q_i \supset \cdots \supset Q_\infty=Q^+_{\theta_1\rho}.
\]
Applying Lemma~\ref{reverse lemma'} with $r=\rho_{i+1}$, $\rho=\rho_i$ and
\[
q_{i+1}=\alpha \kappa, \quad q_i=\alpha
\]
which means that $q_i=q_0\kappa^i \nearrow \infty$ as $i \to \infty$,  
we get for $i=0,1,\ldots$,
\begin{align}\label{reverse eq.a'}
\left(\fiint_{Q_{i+1}}u^{q_{i+1}}\dxt\right)^{\frac{1}{q_{i+1}}}\leq c^{\frac{1}{\alpha \kappa}}\left[\left(\frac{\rho_i}{\rho_{i+1}}\right)^n\left(\frac{\rho_i}{\rho_i-\rho_{i+1}}\right)^{n+sp}(1+\rho_i^{(1-s)p})\right]^{\frac{1}{q_i}}\left(\fiint_{Q_i}u^{q_i}\dxt\right)^{\frac{1}{q_i}}.
\end{align}
One readily checks the three bounds: 
\[
\frac{\rho_i}{\rho_{i+1}}\leq 2,\qquad \frac{\rho_i}{\rho_i-\rho_{i+1}}\leq \frac{2^{i+1}\theta_2}{\theta_2-\theta_1} \quad \textrm{and}\quad  \rho_i^{(1-s)p} \leq \rho^{(1-s)p}.
\]
These bounds and~\eqref{reverse eq.a'} give that
\begin{align}\label{reverse eq.b'}
Y_{i+1}&\leq c^{\frac{1}{q_{i+1}}}\left[2^n\left(\frac{2^{i+1}\theta_2}{\theta_2-\theta_1}\right)^{n+sp}(1+\rho^{(1-s)p})\right]^{\frac{1}{q_i}}Y_i \notag\\[4mm]
&\leq c^{\frac{1}{q_{i+1}}}4^{(n+sp)\frac{i}{q_i}}\left[\left(\frac{\theta_2}{\theta_2-\theta_1}\right)^{n+sp}\left(1+\rho^{(1-s)p}\right)\right]^{\frac{1}{q_i}}Y_i,
\end{align}
where we again used the shorthand notation
\[
Y_i:=\left(\fiint_{Q_{i}}u^{q_{i}}\dxt\right)^{\frac{1}{q_{i}}}\qquad \forall i =0,1,2,\ldots.
\]
Iterating~\eqref{reverse eq.b'}, we obtain that, for all $k \in \mathbb{N}$,
\begin{equation}\label{reverse eq.c'}
Y_k\leq c^{S(k)}4^{(n+sp)T(k)}\left[\left(\frac{\theta_2}{\theta_2-\theta_1}\right)^{n+sp}\left(1+\rho^{(1-s)p}\right)\right]^{S(k)}Y_0,
\end{equation}
where
\[
S(k):=\sum_{i=0}^k\frac{1}{q_i},\qquad T(k):=\sum_{i=0}^{k-1}\frac{i}{q_i}.
\]
As argued before, since
\[
\lim \limits_{k \to \infty}S(k)=\dfrac{\kappa}{q_0(\kappa-1)}=\dfrac{n+p}{q_0p}
\]
and
\[
\lim \limits_{k \to \infty}T(k)=\dfrac{\kappa}{q_0(\kappa-1)^2},
\]
passing  to the limit  $k \to \infty $ in~\eqref{reverse eq.c'}, we arrive at
\begin{align*}
\sup_{Q_{\theta_1\rho}(z_0)}u=\lim_{k \to \infty}Y_k &\leq \left[C\left(\frac{\theta_2}{\theta_2-\theta_1}\right)^{n+sp}\left(1+\rho^{(1-s)p}\right)\right]^{\frac{n+p}{q_0p}}\left(\fiint_{Q_{\theta_2 \rho}}u^{q_{0}}\dxt\right)^{\frac{1}{q_{0}}}
%\\[4mm]
%&\leq \left(C\frac{1+\rho^{(1-s)p}}{(\theta_2-\theta_1)^{n+sp}}\left(\frac{\tau}{\theta_2}\right)^{p}\right)^{\frac{n+p}{q_0p}}\left(\fiint_{Q_{\tau\rho}}u^{q_{0}}\dxt\right)^{\frac{1}{q_{0}}}
\end{align*}
with $C=C(n,s,p,c_0,c_1,\Lambda)$. We now choose $q_0=p$ and $\beta \in (0, p)$. Young's inequality gives
\begin{align}\label{reverse eq.d'}
\sup_{Q^+_{\theta_1\rho}(z_0)}u &\leq \left[C\left(\frac{\theta_2}{\theta_2-\theta_1}\right)^{n+sp}\left(1+\rho^{(1-s)p}\right)\right]^{\frac{n+p}{p^2}}\left(\fiint_{Q^+_{\theta_2\rho}}u^{p}\dxt\right)^{\frac{1}{p}} \notag\\[4mm]
&\leq \sup_{Q^+_{\theta_2\rho}(z_0)}u^{\frac{\beta}{p}}\left[C\left(\frac{\theta_2}{\theta_2-\theta_1}\right)^{n+sp}\left(\frac{\tau}{\theta_2}\right)^{p}\left(1+\rho^{(1-s)p}\right)\right]^{\frac{n+p}{p^2}}\left(\fiint_{Q^+_{\tau \rho}}u^{p-\beta}\dxt\right)^{\frac{1}{p}}  \notag\\[4mm]
&\leq \frac{1}{2}\sup_{Q^+_{\theta_2\rho}(z_0)}u+\left(C\sigma^{-(1-s)p}\frac{1+\rho^{(1-s)p}}{(\theta_2-\theta_1)^{n+sp}}\right)^{\frac{n+p}{p(p-\beta)}}\left(\fiint_{Q^+_{\tau \rho}}u^{p-\beta}\dxt\right)^{\frac{1}{p-\beta}},
\end{align}
where in the last line we carefully estimated that 
\[
\theta_2^{n+sp} \left(\frac{\tau}{\theta_2}\right)^{p} \leq \theta_2^{n+sp} \frac{\tau^p}{\theta_2^{n+p}} \leq \theta_2^{-(1-s)p} \leq \sigma^{-(1-s)p}.
\]
Set $\gamma:=p-\beta>0$ and define, for $s\in [\sigma,\tau]$, $\mathscr{Z}(s):=\sup\limits_{Q_{s\rho}(z_0)}u$. Since $\mathscr{Z}(s)$ is nonnegative and bounded by~\eqref{reverse eq.d'}$_1$, Lemma~\ref{iteration lemma} with 
\[
A\equiv \left(C\sigma^{-(1-s)p}\left(1+\rho^{(1-s)p}\right)\right)^{\frac{n+p}{p\gamma}}\left(\fiint_{Q^+_{\tau \rho}}u^{\gamma}\dxt\right)^{\frac{1}{\gamma}}
\]
yields that
\[
\sup_{Q^+_{\sigma\rho}(z_0)}u \leq C(n,s,p,c_0,c_1,\Lambda,\gamma)\left(\sigma^{-(1-s)p}\frac{1+\rho^{(1-s)p}}{(\tau-\sigma)^{n+sp}}\right)^{\frac{n+p}{p\gamma}}\left(\fiint_{Q^+_{\tau \rho}}u^{\gamma}\dxt\right)^{\frac{1}{\gamma}},
\]
finishes the proof.
\end{proof}
We note that, removing the positivity condition $u \geq m>0$ and adding the restriction $p \geq 2$, we can deduce the local boundedness with tail as follows. The proof can be similarly as~\cite[Theorem 1.1]{Nakamura}.
\begin{thm}[Local boundedness with tail for subsolutions]\label{local bounds with tail}
Let $2 \leq p <\infty $, $s \in (0,1)$ and let $u$ be a possibly sign-changing weak subsolution to~\eqref{maineq} in the sense of Definition~\ref{weak sol}. Then, for every forward space-time cylinder $Q^+_{\rho}(z_0) \Subset \Omega_T$ and every $\delta \in (0,1]$ there exists a constant $c\equiv c(n,s,p,c_0,c_1,\Lambda)$ such that
\[
\sup_{Q^+_{\frac{\rho}{2}}(z_0)} u \leq \delta \,\mathrm{Tail}_\infty \left(u_+, Q^+_{\frac{\rho}{2}}(z_0)\right)+c(1+\rho^{(1-s)p}+\rho^p\delta^{1-p})^{\frac{n+p}{p^2}}\left(\,\fiint_{Q^+_\rho(z_0)}u_+^p\dxt\right)^{\frac{1}{p}},
\]
where $u_+:=\max\{u,0\}$ denotes the positive part of $u$.
\end{thm}

\section{Proof of Theorem~\ref{mainthm}}\label{Sect. 5}
We are now ready to prove Theorem~\ref{mainthm}.
\begin{proof}
We shall prove Theorem~\ref{mainthm} step by step.
\medskip

\emph{Step 1: Setting.}\quad For a supersolution $u$ to~\eqref{maineq} satisfying $u>m \geq 0$ in $\bR^n \times (0,T)$, we set
\[
\beta:=\int_{B_\rho(x_0)}\chi^p(x)\log u(x,t_0)\dx
\]
with the cut-off function appearing on Lemma~\ref{log-type est.}, and define 
\[
v^-:=ue^{-\beta} \quad \textrm{and} \quad  v^+:=u^{-1}e^\beta.
\]
By Lemma~\ref{supsub}, $v^-$ and $v^+$ are weak supersolution to~\eqref{maineq} and  subsolution to~\eqref{maineq} replaced $\mathbf{A}$ by $\widetilde{\mathbf{A}}$, respectively. Thus, applying Lemma~\ref{log-type est.} with $\sigma=\tau$, we infer that
\begin{align}\label{mainthm eq.1}
\Big|Q^-_{\tau\rho}(z_0) \cap \left\{\log v^->\lambda\right\} \Big|&\leq \frac{\overline{C}}{\lambda^{p-1}}\Big|Q^-_{\tau\rho}(z_0)\Big| \notag
\intertext{and}
\Big|Q^+_{\tau\rho}(z_0) \cap \left\{\log v^+>\lambda\right\} \Big|&\leq \frac{\overline{C}}{\lambda^{p-1}}\Big|Q^+_{\tau\rho}(z_0)\Big|,
\end{align}
where $\lambda>0$ and $\overline{C}=\overline{C}(n,s,p,\tau)$.
\medskip

\emph{Step 2: Bombieri-Giusti type estimate.}\quad Following the idea from~\cite[Theorem 1.4]{BG}, we derive the Bombieri-Giusti type estimate as follows: given $\delta \in (0,1)$, let us set 
\[
\Phi(\tau):=\log\left(\fiint_{Q^+_{\tau \rho}}(v^+)^{q^+}\dxt\right)^{\frac{1}{q}}, \quad \Psi(\tau):=\log\left(\fiint_{Q^-_{\tau \rho}}(v^-)^{q^-}\dxt\right)^{\frac{1}{q^-}}
\]
for every $\delta \leq \tau \leq \frac{1+\delta}{2}$,  $q^+>0$ and $0<\gamma<q^-<\frac{n+p}{n}(p-1)$. By H\"{o}lder's inequality and~\eqref{mainthm eq.1}$_2$ with $\lambda=\Phi(\tau)/2>0$, we have, for $0<\gamma<q$, 
\begin{align*}
\fiint_{Q^+_{\tau \rho}}(v^+)^{q^+}\dxt&=\frac{1}{|Q^+_{\tau \rho}|}\left[\iint_{Q^+_{\tau \rho} \cap \{v^+ \,\leq \,\exp(\Phi(\tau)/2)\}}(v^+)^\gamma\dxt+\iint_{Q^+_{\tau \rho} \cap \{v^+\, > \,\exp(\Phi(\tau)/2)\}}(v^+)^\gamma\dxt\right] \\[3mm]
&\leq \exp\left(\gamma \Phi(\tau)/2\right)+\left(\fiint_{Q^+_{\tau \rho}}(v^+)^{q^+}\dxt\right)^{\frac{\gamma}{q^+}}\left(\frac{\big|Q^+_{\tau \rho}\cap\{v^+>\exp(\Phi(\tau)/2)\}\big|}{|Q^+_{\tau \rho}|}\right)^{1-\frac{\gamma}{q^+}} \\[3mm]
&\!\!\!\stackrel{\eqref{mainthm eq.1}_2}{\leq}\exp\left(\gamma \Phi(\tau)/2\right)+\exp\left(\gamma \Phi(\tau)\right)\left(\frac{\overline{C}}{(\Phi(\tau)/2)^{p-1}}\right)^{1-\frac{\gamma}{q^+}}.
\end{align*}
Choosing $\Phi(\tau)$ enough large so that
\[
0<\log\left(\frac{(\Phi(\tau)/2)^{p-1}}{\overline{C}}\right) \leq q^+ \Phi(\tau) \iff \exp(-q^+\Phi) \leq \frac{\overline{C}}{(\Phi(\tau)/2)^{p-1}}<1,
\]
which in turn implies
\[
\fiint_{Q^+_{\tau \rho}}(v^+)^{q^+}\dxt \leq \exp\left(\gamma \Phi(\tau)/2\right)\left[1+\exp(\gamma\Phi(\tau)/2)\left(\frac{\overline{C}}{(\Phi(\tau)/2)^{p-1}}\right)^{1-\frac{\gamma}{q^+}}\right].
\]
Since by selecting the exponent $\gamma$ such that
\begin{align}\label{mainthm eq.2}
&\frac{2}{3}q^+ \geq \gamma=\frac{2}{3}\Phi(\tau)^{-1}\log\left(\frac{(\Phi(\tau)/2)^{p-1}}{\overline{C}}\right) \notag \\[3mm]
\iff &\exp(\gamma\Phi(\tau)/2)=\left(\frac{\overline{C}}{(\Phi(\tau)/2)^{p-1}}\right)^{-\frac{1}{3}},\qquad \frac{2}{3}-\frac{\gamma}{q^+} \geq 0 
\end{align}
it follows that
\[
\exp(\gamma\Phi(\tau)/2)\left(\frac{\overline{C}}{(\Phi(\tau)/2)^{p-1}}\right)^{1-\frac{\gamma}{q^+}}=\left(\frac{\overline{C}}{(\Phi(\tau)/2)^{p-1}}\right)^{\frac{2}{3}-\frac{\gamma}{q^+}} \leq 1,
\]
we therefore gain
\begin{equation*}
\fiint_{Q^+_{\tau \rho}}(v^+)^{q^+}\dxt \leq 2\exp(\gamma \Phi(\tau)/2)
\end{equation*}
with $\gamma$ being as in~\eqref{mainthm eq.2}. 
\medskip

\emph{Step 3: Conclusion.} \quad A careful re-reading of the proof of Lemma~\ref{local bounds} shows that the lower and upper bounds for $\sigma$ and $\tau$ can be replaced by $\delta$ and $\frac{1+\delta}{2}$, respectively. Now we apply Lemma~\ref{local bounds} thereby obtaining, whenever $\delta \leq \sigma<\tau\leq\frac{1+\delta}{2}$
\begin{align*}
\exp\Phi(\sigma)&=\left(\fiint_{Q^+_{\sigma \rho}}(v^+)^{q^+}\dxt \right)^{\frac{1}{q^+}} \leq \sup_{Q^+_{\sigma \rho}}v^+ \\[3mm]
&\leq c \left(\sigma^{-(1-s)p}\frac{1+\rho^{(1-s)p}}{(\tau-\sigma)^{n+sp}}\right)^{\frac{n+p}{p\gamma}}\left(\fiint_{Q^+_{\tau \rho}}(v^+)^{\gamma}\dxt \right)^{\frac{1}{\gamma}} \\[3mm]
&\leq \left(c\delta^{-(1-s)p}\frac{1+\rho^{(1-s)p}}{(\tau-\sigma)^{n+sp}}\right)^{\frac{n+p}{p\gamma}}\left(2\exp(\gamma \Phi(\tau)/2)\right)^{\frac{1}{\gamma}},
\end{align*}
which together with the fact that $\rho \leq 1$ and~\eqref{mainthm eq.2} yields that
\begin{align*}
\Phi(\sigma) &\leq \frac{1}{\gamma} \left(\frac{c\delta^{-(1-s)p}}{(\tau-\sigma)^{n+sp}}\right)^{\frac{n+p}{p}}+\frac{\Phi(\tau)}{2} \\[3mm]
&=\frac{\Phi(\tau)}{2}\left[1+\frac{3(n+p)}{p}\frac{\log \left(\dfrac{c\delta^{-(1-s)p}}{(\tau-\sigma)^{n+sp}}\right)}{\log\left(\dfrac{(\Phi(\tau)/2)^{p-1}}{\overline{C}}\right)}\right].
\end{align*}
We further select a suitably large $\Phi(\tau)$ such that
\[
\frac{3(n+p)}{p}\frac{\log \left(\dfrac{c\delta^{-(1-s)p}}{(\tau-\sigma)^{n+sp}}\right)}{\log\left(\dfrac{(\Phi(\tau)/2)^{p-1}}{\overline{C}}\right)} \leq \frac{1}{2} \iff \Phi(\tau) \geq \underbrace{2\,\overline{C}^{\frac{1}{p-1}}}_{=:\,\overline{C}_0}\left(\frac{c\delta^{-(1-s)p}}{(\tau-\sigma)^{n+sp}}\right)^{\frac{6(n+p)}{p(p-1)}}
\]
and therefore, we gain, whenever $\delta \leq \sigma<\tau\leq \frac{1+\delta}{2}$
\begin{equation*}
\Phi(\sigma) \leq \frac{3}{4}\Phi(\tau)
\end{equation*}
provided $\Phi(\tau) \geq \overline{C}_0\left(\dfrac{c\delta^{-(1-s)p}}{(\tau-\sigma)^{n+sp}}\right)^{\frac{6(n+p)}{p(p-1)}}$. 
\medskip

In the opposite case $\Phi(\tau) < \overline{C}_0\left(\dfrac{c\delta^{-(1-s)p}}{(\tau-\sigma)^{n+sp}}\right)^{\frac{6(n+p)}{p(p-1)}}$, we infer that
\begin{align*}
\psi(\sigma)=\left(\Phi(\sigma)-\Phi(\tau)\right)+\Phi(\tau)&=\frac{1}{q^+}\log \left(\frac{\displaystyle \fiint_{Q^+_{\sigma\rho}}(v^+)^{q^+}\dxt}{\displaystyle\fiint_{Q^+_{\tau\rho}}(v^+)^{q^+}\dxt}\right)+\Phi(\tau) \\[3mm]
&\leq \frac{1}{q^+}\log \left(\frac{\displaystyle\frac{1}{|Q^+_{\sigma\rho}|}\iint_{Q^+_{\tau\rho}}(v^+)^{q^+}\dxt}{\displaystyle\frac{1}{|Q^+_{\tau\rho}|}\iint_{Q^+_{\tau\rho}}(v^+)^{q^+}\dxt}\right)+\Phi(\tau) \\[3mm]
&=\frac{n+p}{q^+}\log\left(\frac{\tau}{\sigma}\right)+\Phi(\tau) \\[3mm]
&\leq \frac{n+p}{q^+}\log\left(\frac{1+\delta}{2\delta}\right)+\overline{C}_0\left(\dfrac{c\delta^{-(1-s)p}}{(\tau-\sigma)^{n+sp}}\right)^{\frac{6(n+p)}{p(p-1)}}
\end{align*}
holds whenever $\delta \leq \sigma<\tau\leq \frac{1+\delta}{2}$.
Bearing in mind these observations, we conclude that
\[
\Phi(\sigma) \leq \frac{3}{4}\Phi(\tau)+\dfrac{A}{(\tau-\sigma)^{(n+sp)\frac{6(n+p)}{p(p-1)}}}+\frac{n+p}{q^+}\log\left(\frac{1+\delta}{2\delta}\right)
\]
with $A=A(n,s,p,c_0,c_1,\Lambda,\delta)$. As a consequence, applying Lemma~\ref{iteration lemma} with $\mathcal{Z}(\sigma)=\Phi(\sigma)$, we have
\[
\left(\fiint_{Q^+_{\delta \rho}}(v^+)^{q^+}\dxt\right)^{\frac{1}{q^+}}=\Phi(\delta) \leq C(n,s,p,c_0,c_1,\Lambda,\delta)\left[\frac{1}{(1-\delta)^{(n+sp)\frac{6(n+p)}{p(p-1)}}}+\frac{1}{q^+}\log\left(\frac{1+\delta}{2\delta}\right)\right]
\]
Since the exponent $q^+$ is arbitrary, we are allowed to pass to the limit $q^+ \nearrow \infty$ in the above display, and this finally yields that
\begin{equation}\label{mainthm eq.4}
\sup_{Q^+_{\delta\rho}}v^+ \leq C(n,s,p,c_0,c_1,\Lambda,\delta).
\end{equation}
Similarly, for term $\Psi(\tau)$ we estimate that
\[
\fiint_{Q^-_{\tau\rho}}(v^-)^\gamma\dxt \leq 2\exp(\gamma\Psi(\tau)/2)
\]
with suitable choice $\gamma$ and large enough $\Psi$ in the following:
\[
\frac{2}{3}q^- \geq \gamma=\frac{2}{3}\Psi(\tau)^{-1}\log\left(\frac{(\Psi(\tau)/2)^{p-1}}{\overline{C}}\right) \quad \textrm{and}\quad 0<\log\left(\frac{(\Psi(\tau)/2)^{p-1}}{\overline{C}}\right) \leq q^- \Psi(\tau).
\]
Noticing that Lemma~\ref{reverse lemma} is valid for every $0<\gamma<q^-<\frac{n+p}{n}(p-1)$ and repeating the above proof, we gain, for $0<q^-<\frac{n+p}{n}(p-1)$,
\[
\Psi(\sigma) \leq \frac{3}{4}\Psi(\tau)+\dfrac{B}{(\tau-\sigma)^{(n+sp)\frac{6(n+p)}{p(p-1)}}}+\frac{n+p}{q^-}\log\left(\frac{1+\delta}{2\delta}\right)
\]
with $B=B(n,s,p,c_0,c_1,\Lambda,\delta)$. Thus, Lemma~\ref{iteration lemma} implies that
\begin{align}\label{mainthm eq.5}
\left(\fiint_{Q^-_{\delta \rho}}(v^-)^{q^-}\dxt\right)^{\frac{1}{q^-}}=\Psi(\delta) &\leq B(n,s,p,c_0,c_1,\Lambda,\delta)\left[\frac{1}{(1-\delta)^{(n+sp)\frac{6(n+p)}{p(p-1)}}}+\frac{1}{q^-}\log\left(\frac{1+\delta}{2\delta}\right)\right] \notag\\[3mm]
&=:D(n,s,p,c_0,c_1,\Lambda,\delta,q^-)
\end{align}
From definition of $v^\pm$,~\eqref{mainthm eq.4} and~\eqref{mainthm eq.5}, the assertion finally follows.
\end{proof}
%%%%%%
%%%%%%%%%%%%%%%%%%%%%%%%%%%%%%%%%
%%%%%%%%%%%%%%%%%%%%%%%%%%%%%%%%%
\section{Proof of Theorem~\ref{mainthm2}}\label{Sect. 6}
%%%%%%%%%%%%%%%%%%%%%%%%%%%%%%%%%
%%%%%%%%%%%%%%%%%%%%%%%%%%%%%%%%%
In this final section we report the proof of Theorem~\ref{mainthm2}.
\begin{proof}
Let $0<\rho \leq 1$. Given $\sigma \in (0,1)$, from Theorem~\ref{mainthm} with $\delta=\frac{1+\sigma}{2}$ it readily follows that
\[
\left(\fiint_{Q_{\frac{1+\sigma}{2}\rho}^-(z_0)}u^q\dxt\right)^{\frac{1}{q}} \leq C \inf_{Q^+_{\frac{1+\sigma}{2} \rho}(z_0)}u.
\]
Since $u$ is a weak subsolution to~\eqref{maineq}, Proposition~\ref{local bounds} yields that
\[
\sup_{Q^-_{\sigma\rho}(z_0)}u \leq \frac{c}{(1-\sigma)^{(n+sp)\frac{n+p}{nq}}}\left(\fiint_{Q_{\frac{1+\sigma}{2}\rho}^-(z_0)}u^q\dxt\right)^{\frac{1}{q}},
\]
which together with the fact that $\inf \limits_{Q^+_{\frac{1+\sigma}{2} \rho}(z_0)}u \leq \inf \limits_{Q^+_{\sigma\rho}(z_0)}u$ concludes the proof.
\end{proof} 
%%%%%%%%%%%%%%%%%%%%%%%%%%%%%%%%%%%%%%%%%%%%%%%%%%%%%%%%%%%%%%%%%%%%%%%%%%%%%%%%%%%%%%%%%%%%%%%%%%%%%%%%%%%%%%%%%%%%%%%%%%%%%%%%%%%%%%%%%%%%
%%%%%%%%Appendix%%%%%%%%%%%%%%%%%
%%%%%figure number in appendix%%%%%%%%%%%%%%%%%%%%
\makeatletter
\renewcommand{\thesection}{\Alph{section}.\arabic{subsection}}
\makeatother

\appendix

%%%%%figure number in appendix%%%%%%%%%%%%%%%%%%%%
\makeatletter
\renewcommand{\thefigure}{\Alph{section}.\arabic{figure}}
\@addtoreset{figure}{section}
\makeatother
%%%%%%%%%%%%%%%%%%%%%%%%%%%%%%%%%
%%%%%%%%%%%%%%%%%%%%%%%%%%%%%%%%%%%%%%%%
%%%%%%%%%%%%%%%%%%%%%%%%%%%%%%%%%%%%%%%%
%%%%%%%%%%%%%%%%%%%%%%%%%%%%%%%%%%%%%%%%

\section{Proof of Lemma~\ref{Caccioppoli1}}\label{Appendix A}

\begin{proof}[\normalfont \textbf{Proof of Lemma~\ref{Caccioppoli1}}]
Although the argument used here is almost same as Proposition~\ref{log-type Caccioppoli} and Lemma~\ref{supsub} (also~\cite[Proposition 3.1]{Nakamura}), for the sake of completeness we shall nevertheless give the full proof. For this, we divide it into several steps.
\medskip

\emph{Step 1: Beginning.}\quad Let $t_1 \in (t_0-\rho^p,t_0)$ be an arbitrary number. Given $\delta>0$ small, take the cutoff function $\psi_\delta(t)$ with respect to time, defined by
\[
\psi_\delta(t):=\begin{cases}
0,\quad &\,\,\,t \in (t_0-\rho^p, t_1), \\
\frac{1}{\delta}(t-t_1), \quad &\,\,\,t \in [t_1, t_1+\delta),\\
1,\quad &\,\,\,t \in [t_1+\delta, t_0).
\end{cases}
\]
In the weak formulation~\eqref{D2'}, we switch the test function $\varphi$ to $\varphi^p\psi_\delta(t)u^{-\varepsilon}$ with nonnegative $\varphi \in C^\infty_0(Q^-_\rho(z_0))$. Then the evolutional term is divided into two terms as follows:
\begin{align*}
\iint_{\Omega_T}\partial_t[|u|^{p-2}u]_h\varphi^p\psi_\delta(t)u^{-\varepsilon}\dxt&=\iint_{Q^-_\rho(z_0)}\varphi^p\psi_\delta(t)\partial_t[u^{p-1}]_h\left(u^{-\varepsilon}-[u^{p-1}]_h^{-\frac{\varepsilon}{p-1}}\right)\dxt \\[3mm]
&\quad \quad \quad+\iint_{Q^-_\rho(z_0)}\varphi^p\psi_\delta(t)\partial_t[u^{p-1}]_h[u^{p-1}]_h^{-\frac{\varepsilon}{p-1}}\dxt \\[3mm]
&=:\mathbf{I}_1+\mathbf{I}_2,
\end{align*}
where the definition of $\mathbf{I}_1$ and $\mathbf{I}_2$ are obvious form the context. Abbreviating 
\[
f_h:=[u^{p-1}]_h^{\frac{1}{p-1}} \quad \iff \quad [u^{p-1}]_h=f_h^{p-1},
\]
and  Lemma~\ref{mollification lemma}-(ii) and~\eqref{algs'} with $\alpha=1+\frac{p-1}{\varepsilon}$ in Lemma~\ref{Algs'} imply that
\begin{align*}
\mathbf{I}_1&=\iint_{Q^-_\rho(z_0)}\varphi^p\psi_\delta(t)\frac{u^{p-1}-[u^{p-1}]_h}{h}\cdot \frac{[u^{p-1}]_h^{\frac{\varepsilon}{p-1}}-u^\varepsilon}{[u^{p-1}]_h^{\frac{\varepsilon}{p-1}}u^\varepsilon}\dxt \\[3mm]
&=\iint_{Q^-_\rho(z_0)}\varphi^p\psi_\delta(t)\frac{u^{p-1}-f_h^{p-1}}{h}\cdot \frac{f_h^\varepsilon-u^\varepsilon}{f_h^\varepsilon u^\varepsilon}\dxt \leq 0.
\end{align*}
From integration by parts, it follows that
\begin{align*}
\mathbf{I}_2&=\frac{p-1}{p-1-\varepsilon}\iint_{Q^-_\rho(z_0)}\varphi^p\psi_\delta(t)\partial_t[u^{p-1}]_h^{\frac{p-1-\varepsilon}{p-1}}\dxt \\[3mm]
&=-\frac{p-1}{p-1-\varepsilon}\iint_{Q^-_\rho(z_0)}\left(p\varphi^{p-1}\varphi_t\psi_\delta+\varphi^p\psi\delta^\prime \right)[u^{p-1}]_h^{\frac{p-1-\varepsilon}{p-1}}\dxt.
\end{align*}
The preceding estimates above and the fact that $[u^{p-1}]_h^{\frac{1}{p-1}} \to u$ in $L^p(\Omega_T)$ by Lemma~\ref{mollification lemma}-(iii) give that
\begin{align}\label{a1}
&\limsup_{h\searrow 0}\iint_{\Omega_T}\partial_t[|u|^{p-2}u]_h\varphi^p\psi_\delta(t)u^{-\varepsilon}\dxt \notag \\[3mm]
&\leq \limsup_{h\searrow 0}\left(\mathbf{I}_1+\mathbf{I}_2\right) \notag\\[3mm]
&\leq -\frac{p-1}{p-1-\varepsilon}\iint_{Q^-_\rho(z_0)}\left(p\varphi^{p-1}\varphi_t\psi_\delta+\varphi^p\psi_\delta^\prime \right)u^{p-1-\varepsilon}\dxt \notag\\[3mm]
&\leq  -\frac{p-1}{p-1-\varepsilon}\,\dashint_{t_1}^{t_1+\delta}\int_{B_\rho(x_0)}u^{p-1-\varepsilon}\varphi^p\dxt+\frac{p(p-1)}{p-1-\varepsilon}\iint_{Q^-_\rho(z_0)}\varphi^{p-1}|\varphi_t|u^{p-1-\varepsilon}\psi_\delta\dxt.
\end{align}

The spatial term can be dealt with comparatively easily. Indeed, by the use of convergence $[|Du|^{p-2}Du]_h \to |Du|^{p-2}Du$ in $L^{\frac{p}{p-1}}(\Omega_T)$ as $h\searrow 0$ via Lemma~\ref{mollification lemma}-(ii) and Young's inequality with the exponents $\left(\frac{p}{p-1},p\right)$ yields that, for any $\theta>0$, 
\begin{align*}
&\lim_{h\searrow 0}\iint_{\Omega_T}[|Du|^{p-2}Du]_h\cdot D\left(\varphi^p\psi_\delta(t)u^{-\varepsilon}\right)\dxt \\[3mm]
&=\iint_{\Omega_T}|Du|^{p-2}Du\cdot D\left(\varphi^p\psi_\delta(t)u^{-\varepsilon}\right)\dxt \\[3mm]
&=\iint_{\Omega_T}|Du|^{p-2}Du\cdot \left(p\varphi^{p-1}D\varphi\,u^{-\varepsilon}+\varphi^pDu^{-\varepsilon}\right)\dxt\\[3mm]
&\leq p\iint_{Q^-_\rho(z_0)}\Big(|Du|u^{-\frac{\varepsilon+1}{p}}\varphi\psi_\delta^{\frac{1}{p}}\Big)^{p-1}\Big(|D\varphi|u^{-\varepsilon +(\varepsilon+1)\frac{p-1}{p}}\psi_\delta^{\frac{1}{p}}\Big)\dxt\\[3mm]
&\quad \quad \quad \quad -\varepsilon\iint_{\Omega_T}|Du|^pu^{-\varepsilon-1}\varphi^p\psi_\delta\dxt \\[3mm]
&\leq \big[(p-1)\theta-\varepsilon\big]\iint_{\Omega_T}|Du|^pu^{-\varepsilon-1}\varphi^p\psi_\delta\dxt+c(\theta)\iint_{Q^-_\rho(z_0)}|D\varphi|^pu^{-\varepsilon p+(\varepsilon+1)(p-1)}\varphi^p\psi_\delta\dxt,
\end{align*}
which, by taking $\theta=\frac{\varepsilon}{2(p-1)}$, yields in particular that
\begin{align}\label{a2}
&\lim_{h\searrow 0}\iint_{\Omega_T}[|Du|^{p-2}Du]_h\cdot D\left(\varphi^p\psi_\delta(t)u^{-\varepsilon}\right)\dxt \notag\\[3mm]
&\leq -\frac{\varepsilon}{2}\iint_{\Omega_T}|Du|^pu^{-\varepsilon-1}\varphi^p\psi_\delta\dx+c(\varepsilon)\iint_{Q^-_\rho(z_0)}|D\varphi|^pu^{p-1-\varepsilon}\psi_\delta\dxt,
\end{align}
where we have computed that $-\varepsilon p+(\varepsilon+1)(p-1)=p-1-\varepsilon$.
\medskip

\emph{Step 2: Fractional term.}\quad Firstly, we will show that
\begin{align}\label{a3}
&\int_0^T\iint_{\bR^n\times \bR^n}\left[U(x,y,t)K(x,y,t)\right]_h\left(\phi(x,t)-\phi(y,t)\right)\dxyt \notag \\[3mm]
&\to \int_0^T\iint_{\bR^n\times \bR^n}U(x,y,t)K(x,y,t)\left(\phi(x,t)-\phi(y,t)\right)\dxyt 
\end{align}
in the limit $h\searrow 0$, where we abbreviated $\phi\equiv \varphi^p\psi_\delta(t)u^{-\varepsilon}$ for short. In order to prove~\eqref{a3} we consider the quantity
\begin{align*}
(\mathbf{II})_h&:=\left|\int_0^T\iint_{D_\Omega}\Big(\left[U(x,y,t)K(x,y,t)\right]_h-U(x,y,t)K(x,y,t)\Big)\left(\phi(x,t)-\phi(y,t)\right)\dxyt \right| \\[3mm]
&\leq \int_0^T\iint_{\Omega \times \Omega}\big|\cdots \big|\dxyt+2\int_0^T\iint_{\Omega \times \Omega^c}\big|\cdots \big|\dxyt \\[3mm]
&=:(\mathbf{II}_1)_h+2(\mathbf{II}_2)_h,
\end{align*}
with denoting $D_\Omega:=(\bR^n \times \bR^n) \setminus (\Omega^c \times \Omega^c)$, where the definition of $(\mathbf{II}_1)_h$ and $(\mathbf{II}_2)_h$ are clear from the context. We now claim the following: As $h \searrow 0$, 
\begin{equation}\label{a4}
\Big[U(x,y,t)K(x,y,t)|x-y|^{\frac{1}{p}(n+sp)}\Big]_h \to U(x,y,t)K(x,y,t)|x-y|^{\frac{1}{p}(n+sp)}\quad \textrm{in}\quad L^{\frac{p}{p-1}}(\Omega^2_T),
\end{equation}
where $\Omega^2_T:=\Omega \times \Omega \times (0,T)$. Indeed, since \[
\Big\|U(x,y,t)K(x,y,t)|x-y|^{\frac{1}{p}(n+sp)}\Big\|_{L^{\frac{p}{p-1}}(\Omega^2_T)}^{\frac{p}{p-1}}<\infty
\]
 Lemma~\ref{mollification lemma}-(i) with $E=\Omega^2$ implies~\eqref{a4}. Therefore, using~\eqref{a4}, H\"{o}lder's inequality and Lemma~\ref{Wsp ineq.}, we have
 \begin{align*}
 (\mathbf{II}_1)_h &=\int_0^T\iint_{\Omega \times \Omega}\Bigg|\left[U(x,y,t)K(x,y,t)|x-y|^{\frac{1}{p}(n+sp)}\right]_h-U(x,y,t)K(x,y,t)|x-y|^{\frac{1}{p}(n+sp)}\Bigg| \\[3mm]
 &\quad \quad \quad \quad \quad \times \frac{\left|\phi(x,t)-\phi(y,t)\right|}{|x-y|^{\frac{1}{p}(n+sp)}}\dxyt \\[3mm]
 &\leq c\left(\int_0^T\iint_{\Omega \times \Omega}\Bigg|\left[UK(x,y,t)|x-y|^{\frac{1}{p}(n+sp)}\right]_h-UK(x,y,t)|x-y|^{\frac{1}{p}(n+sp)}\Bigg|^{\frac{p}{p-1}}\dxyt\right)^{\frac{p-1}{p}} \\[3mm]
 &\quad \quad \quad \quad \quad \times \left(\int_{\Omega_T}|D\phi(x)|^p\dxt\right)^{\frac{1}{p}} \\[3mm]
 &\to 0
 \end{align*}
as $h\searrow 0$, where we used the shorthand notation $UK(x,y,t)=U(x,y,t)K(x,y,t)$. 
\medskip

We next prove that $ (\mathbf{II}_1)_h \to 0$.  Take a ball $B_R\equiv B_R(0)$ satisfying $B_R \Supset \Omega$. By $\phi(y,t)=0$ for any $(y,t) \in (B_R\setminus \Omega)\times (0,T)$ and H\"{o}lder's inequality, we obtain 
{\small
\begin{align}\label{a5}
(\mathbf{II}_2)_h(R) &:=\int_0^T\iint_{\Omega \times (B_R\setminus \Omega)}\Big|\left[UK(x,y,t)|x-y|^{\frac{1}{p}(n+sp)}\right]_h-UK(x,y,t)|x-y|^{\frac{1}{p}(n+sp)}\Big| \notag\\[3mm]
 &\quad \quad \quad \quad \quad \times \frac{\left|\phi(x,t)\right|}{|x-y|^{\frac{1}{p}(n+sp)}}\dxyt \notag\\[3mm]
 &\leq \left(\int_0^T\iint_{\Omega \times (B_R\setminus \Omega)} \frac{\left|\phi(x,t)\right|^p}{|x-y|^{n+sp}}\dxyt\right)^{\frac{1}{p}} \notag\\[3mm]
 &\times \left(\int_0^T\iint_{\Omega \times (B_R\setminus \Omega)}\Bigg|\left[UK(x,y,t)|x-y|^{\frac{1}{p}(n+sp)}\right]_h-UK(x,y,t)|x-y|^{\frac{1}{p}(n+sp)}\Bigg|^{\frac{p}{p-1}}\dxyt\right)^{\frac{p-1}{p}}\notag\\[3mm]
 &\leq (\mathbf{III})^{\frac{1}{p}}(\mathbf{IV})^{\frac{p-1}{p}},
\end{align}
}

\noindent
where
\[
\mathbf{III}:=\int_0^T\iint_{\Omega \times (B_R\setminus \Omega)} \frac{\left|\phi(x,t)\right|^p}{|x-y|^{n+sp}}\dxyt
\]
and
\[
\mathbf{IV}:=\int_0^T\iint_{\Omega \times (B_R\setminus \Omega)}\Big|\left[UK(x,y,t)|x-y|^{\frac{1}{p}(n+sp)}\right]_h-UK(x,y,t)|x-y|^{\frac{1}{p}(n+sp)}\Big|^{\frac{p}{p-1}}\dxyt.
\]
Since
\[
|x-y| \geq d_R:=\mathrm{dist}\left(\supp\,\phi(\cdot,t),\, B_R \setminus\Omega \right)\qquad \forall (x,y) \in \supp\,\phi(\cdot,t) \times (B_R\setminus \Omega)
\]
and $\supp\,\phi=Q_\rho^-(z_0)$, 
we estimate
 \begin{align}\label{a6}
 \mathbf{III}&=\iint_{Q_\rho^-(z_0)}\left(\int_{B_R\setminus \Omega}\frac{\dy}{|x-y|^{n+sp}} \right)|\phi(x,t)|\dxt \notag\\[3mm]
 &\leq m^{-\varepsilon}\|\varphi\|_{L^\infty(Q_\rho^-(z_0))}|Q_\rho^-(z_0)|\left(\int_{B_{d_R}(x)^c}\frac{\dy}{|x-y|^{n+sp}} \right)\notag\\[3mm]
 &=m^{-\varepsilon}\|\varphi\|_{L^\infty(Q_\rho^-(z_0))}|Q_\rho^-(z_0)|\cdot \frac{c(n)}{sp}d_R^{-sp}.
 \end{align}
 Next, we estimate the integrand of $\mathbf{IV}$:
 \begin{align*}
 &\Bigg|\left[UK(x,y,t)|x-y|^{\frac{1}{p}(n+sp)}\right]_h-UK(x,y,t)|x-y|^{\frac{1}{p}(n+sp)}\Bigg|^{\frac{p}{p-1}} \\[3mm]
 &\leq c(p)\left(\left|\left[UK(x,y,t)|x-y|^{\frac{1}{p}(n+sp)}\right]_h\right|^{\frac{p}{p-1}}+\left|UK(x,y,t)|x-y|^{\frac{1}{p}(n+sp)}\right|^{\frac{p}{p-1}}\right).
 \end{align*}
 By the use of H\"older's inequality and the assumption~\eqref{S3}, we infer that
 \begin{align*}
 \left|\left[UK(x,y,t)|x-y|^{\frac{1}{p}(n+sp)}\right]_h\right|&=\left|\frac{1}{h}\int_0^te^{\frac{s-t}{h}}UK(x,y,s)|x-y|^{\frac{1}{p}(n+sp)}\ds\right| \\[3mm]
 &\leq \Lambda \cdot \frac{1}{h}\int_0^te^{\frac{s-t}{h}}\frac{|U(x,y,s)|}{|x-y|^{\frac{p-1}{p}(n+sp)}}\ds \\[3mm]
 &\leq \Lambda \left(\frac{1}{h}\int_0^te^{\frac{s-t}{h}}\ds\right)^{\frac{1}{p}}\left(\frac{1}{h}\int_0^te^{\frac{s-t}{h}}\frac{|U(x,y,s)|^{\frac{p}{p-1}}}{|x-y|^{n+sp}}\ds\right)^{\frac{p-1}{p}} \\[3mm]
 &=\Lambda (1-e^{-\frac{t}{h}})^{\frac{1}{p}}\left[\frac{|U(x,y,t)|^{\frac{p}{p-1}}}{|x-y|^{n+sp}}\right]_h^{\frac{p-1}{p}},
 \end{align*}
 that is, 
 \[
 \left|\left[UK(x,y,t)|x-y|^{\frac{1}{p}(n+sp)}\right]_h\right|^{\frac{p}{p-1}}\leq \Lambda^{\frac{p}{p-1}} (1-e^{-\frac{t}{h}})^{\frac{1}{p-1}}\left[\frac{|U(x,y,t)|^{\frac{p}{p-1}}}{|x-y|^{n+sp}}\right]_h.
 \]
 Similarly, we have
 \[
 \left|UK(x,y,t)|x-y|^{\frac{1}{p}(n+sp)}\right|^{\frac{p}{p-1}}\leq \Lambda^{\frac{p}{p-1}}\frac{|U(x,y,t)|^{\frac{p}{p-1}}}{|x-y|^{n+sp}}.
 \]
 Collecting the preceding estimates above, the integrand of $\mathbf{IV}$ is estimated as
 \begin{align*}
 &\Bigg|\left[UK(x,y,t)|x-y|^{\frac{1}{p}(n+sp)}\right]_h-UK(x,y,t)|x-y|^{\frac{1}{p}(n+sp)}\Bigg|^{\frac{p}{p-1}} \\[3mm]
 & \leq c\Lambda^{\frac{p}{p-1}} \left(\left[\frac{|U(x,y,t)|^{\frac{p}{p-1}}}{|x-y|^{n+sp}}\right]_h+\frac{|U(x,y,t)|^{\frac{p}{p-1}}}{|x-y|^{n+sp}}\right) \\[3mm]
 &=c\Lambda^{\frac{p}{p-1}} \left(\left[\frac{\big|u(x,t)-u(y,t)\big|^{p}}{|x-y|^{n+sp}}\right]_h+\frac{\big|u(x,t)-u(y,t)\big|^{p}}{|x-y|^{n+sp}}\right) \\[3mm]
 &=:c\Lambda^{\frac{p}{p-1}}W_h(x,y,t).
 \end{align*}
 Since $W_h \in L^1(\Omega \times (B_R \setminus \Omega )\times (0,T))$ holds true, Lemma~\ref{mollification lemma}-(i) applied with $E=\Omega \times (B_R \setminus \Omega )$ yields that
 \[
 W_h \to W:=2\frac{\big|u(x,t)-u(y,t)\big|^{p}}{|x-y|^{n+sp}}\quad \textrm{in}\,\,\,L^1\left(\Omega \times (B_R\setminus \Omega)\times (0,T)\right)
 \]
 as $h\searrow 0$. Therefore using the dominated convergence theorem, we conclude that
 \begin{equation}\label{a7}
 \lim \limits_{h\searrow 0}\mathbf{IV}=0.
 \end{equation}
 Merging~\eqref{a6} with~\eqref{a7} in~\eqref{a5} and, subsequently, sending $R \nearrow \infty$, we finally arrive at the resulting convergence~\eqref{a3}.
 \medskip
 
 \emph{Step 3: Taking the limit as $h\searrow 0$ and $\delta\searrow 0$.}\quad Since by $u \geq m>0$ in $\bR^n \times (t_0-\rho^p,t_0)$, we have
 \[
\lim_{h\searrow 0}\int_\Omega |u|^{p-2}u(0)\left(\frac{1}{h}\int_0^Te^{-\frac{s}{h}}\phi(x,s)\ds\right)\dx=0
 \]
 with $\phi=\varphi^p\psi_\delta(t)u^{-\varepsilon}$. Therefore, combining this with the observations~\eqref{a1}--\eqref{a3} and passing to the limit as $h\searrow 0$ in the weak formulation~\eqref{D2'} with the testing function $\varphi^p\psi_\delta(t)u^{-\varepsilon}$, we gain
 \begin{align*}
 &\frac{p-1}{p-1-\varepsilon}\,\dashint_{t_1}^{t_1+\delta}\int_{B_\rho(x_0)}u^{p-1-\varepsilon}\varphi^p\dxt+\frac{\varepsilon}{2}\iint_{Q^-_\rho(z_0)}|Du|^pu^{-\varepsilon-1}\varphi^p\psi_\delta\dxt\\[3mm]
 &\leq \frac{p(p-1)}{p-1-\varepsilon}\iint_{Q^-_\rho(z_0)}\varphi^{p-1}|\varphi_t|u^{p-1-\varepsilon}\psi_\delta\dxt+c\iint_{Q^-_\rho(z_0)}|D\varphi|^pu^{p-1-\varepsilon}\psi_\delta\dxt
\\[3mm]
&\quad +\int_{t_1}^{t_0}\iint_{\bR^n\times \bR^n}UK(x,y,t)\left(\varphi^p(x,t)u(x,t)^{-\varepsilon}-\varphi^p(y,t)u(y,t)^{-\varepsilon}\right)\psi_\delta(t)\dxyt.
\end{align*}
Passing to the limit $\delta\searrow 0$ combined with Lebegue's differential theorem and the dominated convergence theorem implies that
 \begin{align}\label{a8}
 &\int_{B_\rho(x_0)\times \{t_1\}}u^{p-1-\varepsilon}\varphi^p\dx+\int_{t_1}^{t_0}\int_{B_\rho(x_0)}|Du|^pu^{-\varepsilon-1}\varphi^p\dxt \notag\\[3mm]
 &\leq \left(p+\tfrac{2p(p-1)}{\varepsilon(p-1-\varepsilon)}\right)\int_{t_1}^{t_0}\int_{B_\rho(x_0)}\varphi^{p-1}|\varphi_t|u^{p-1-\varepsilon}\dxt  \notag\\[3mm]
 &\quad \quad \quad +c\left(\tfrac{p-1-\varepsilon}{p-1}+\tfrac{2}{\varepsilon}\right) \int_{t_1}^{t_0}\int_{B_\rho(x_0)}|D\varphi|^pu^{p-1-\varepsilon}\dxt
\notag\\[3mm]
&\quad +c\left(\tfrac{p-1-\varepsilon}{p-1}+\tfrac{2}{\varepsilon}\right)\int_{t_1}^{t_0}\iint_{\bR^n\times \bR^n}UK(x,y,t)\left(\varphi^p(x,t)u(x,t)^{-\varepsilon}-\varphi^p(y,t)u(y,t)^{-\varepsilon}\right)\dxyt.
\end{align}
Since $\frac{p-1-\varepsilon}{p-1}<1<p$ and $\frac{p-1}{p-1-\varepsilon}>1$ the constant
\[
C(\varepsilon,p):=c\left(p+\frac{2p(p-1)}{\varepsilon(p-1-\varepsilon)}\right)
\]
bounds the all constants appearing on the right-hand side of~\eqref{a8}, and we know that $C(\varepsilon,p)$ blows up as $\varepsilon\searrow 0$ or $\varepsilon \nearrow p-1$. In the first term on the left-hand side of~\eqref{a8} we take the supremum over $t_1 \in (t_0-\rho^p, t_0)$, while in the others we let $t_1\searrow t_0-\rho^p$. This finally leads to
 \begin{align}\label{a9}
 &\sup_{t_1 \in (t_0-\rho^p,t_0)}\int_{B_\rho(x_0)\times \{t_1\}}u^{p-1-\varepsilon}\varphi^p\dx+\iint_{Q_\rho^-(z_0)}|Du|^pu^{-\varepsilon-1}\varphi^p\dxt \notag\\[3mm]
 &\leq C\iint_{Q_\rho^-(z_0)}\varphi^{p-1}|\varphi_t|u^{p-1-\varepsilon}\dxt+C\iint_{Q_\rho^-(z_0)}|D\varphi|^pu^{p-1-\varepsilon}\dxt
\notag\\[3mm]
&\quad +C\int_{t_0-\rho^p}^{t_0}\iint_{\bR^n\times \bR^n}U(x,y,t)K(x,y,t)\left(\varphi^p(x,t)u(x,t)^{-\varepsilon}-\varphi^p(y,t)u(y,t)^{-\varepsilon}\right)\dxyt.
\end{align}

\medskip

\emph{Step 4: Conclusion}.\quad In this final step, we are ready to conclude the whole proof, again by estimating the fractional term appearing on the right-hand side of~\eqref{a9}. For this, we estimate separately as follows:
\begin{align*}
\mathbf{V}&:=\int_{t_0-\rho^p}^{t_0}\iint_{\bR^n\times \bR^n}U(x,y,t)K(x,y,t)\left(\varphi^p(x,t)u(x,t)^{-\varepsilon}-\varphi^p(y,t)u(y,t)^{-\varepsilon}\right)\dxyt
\\[3mm]
&=\int_{t_0-\rho^p}^{t_0}\iint_{B_\rho(x_0) \times B_\rho(x_0) }(\cdots)\dxyt+2\int_{t_0-\rho^p}^{t_0}\iint_{B_\rho(x_0) \times (\bR^n \setminus B_\rho(x_0) )}(\cdots)\dxyt \\[3mm]
&=:\mathbf{V}_1+2\mathbf{V}_2,
\end{align*}
with the obvious meaning of $\mathbf{V}_1$ and $\mathbf{V}_2$. Applying Lemma~\ref{fractional est.} with $a=u(y,t)$, $\tau_1=\varphi(y,t)$ and $b=u(y,t)$, $\tau_2=\varphi(x,t)$, the integrand of $\mathbf{V}_1$ is estimated as
\begin{align*}
&U(x,y,t)K(x,y,t)\left(\varphi^p(x,t)u(x,t)^{-\varepsilon}-\varphi^p(y,t)u(y,t)^{-\varepsilon}\right)\\[3mm]
&\leq -\Lambda c(p) \zeta(\varepsilon)\frac{\Big|\varphi(x,t)u(x,t)^{\frac{\alpha}{p}}-\varphi(y,t)u(y,t)^{\frac{\alpha}{p}}\Big|^p}{|x-y|^{n+sp}} \\[3mm]
&\quad +\Lambda\left(\zeta(\varepsilon)+1+\varepsilon^{-(p-1)}\right)\big|\varphi(x,t)-\varphi(y,t)\big|^p\cdot \frac{u(x,t)^\alpha+u(y,t)^\alpha}{|x-y|^{n+sp}}
\end{align*}
and therefore 
{\small
\begin{align*}
\mathbf{V}_1 &\leq  -\Lambda c(p) \zeta(\varepsilon)\int_{t_0-\rho^p}^{t_0}\iint_{B_\rho(x_0) \times B_\rho(x_0)}\frac{\Big|\varphi(x,t)u(x,t)^{\frac{\alpha}{p}}-\varphi(y,t)u(y,t)^{\frac{\alpha}{p}}\Big|^p}{|x-y|^{n+sp}} \dxyt \\[3mm]
&\quad +\Lambda\left(\zeta(\varepsilon)+1+\varepsilon^{-(p-1)}\right)\int_{t_0-\rho^p}^{t_0}\iint_{B_\rho(x_0) \times B_\rho(x_0)}\frac{\big|\varphi(x,t)-\varphi(y,t)\big|^p\left(u(x,t)^\alpha+u(y,t)^\alpha\right)}{|x-y|^{n+sp}}\dxyt.
\end{align*}
}

\noindent
On the other hand, using the positivity of $u$, the term $\mathbf{V}_2$ is estimated as
\begin{align*}
\mathbf{V}_2&\leq \Lambda \int_{t_0-\rho^p}^{t_0}\iint_{B_\rho(x_0) \times (\bR^n \setminus B_\rho(x_0))}\frac{U(x,y,t)}{|x-y|^{n+sp}}\varphi^p(x,t)u(x,t)^{-\varepsilon}\dxyt \\[4mm]
&=\Lambda \int_{t_0-\rho^p}^{t_0}\iint_{B_\rho(x_0) \times (\bR^n \setminus B_\rho(x_0)) \cap \{u(x,t) \,\geq \,u(y,t)\}}(\cdots)\dxyt \\[4mm]
&\quad \quad \quad +\Lambda\underbrace{\int_{t_0-\rho^p}^{t_0}\iint_{B_\rho(x_0) \times (\bR^n \setminus B_\rho(x_0)) \cap \{u(x,t) \,<\,u(y,t)\}}(\cdots)\dxyt}_{<0} \\[4mm]
&\leq \Lambda \int_{t_0-\rho^p}^{t_0}\iint_{B_\rho(x_0) \times (\bR^n \setminus B_\rho(x_0)) \cap \{u(x,t) \,\geq \,u(y,t)\}}\frac{(u(x,t)-u(y,t))^{p-1}}{|x-y|^{n+sp}}\varphi^p(x,t)u(x,t)^{-\varepsilon}\dxyt \\[4mm]
&\leq  \Lambda \int_{t_0-\rho^p}^{t_0}\iint_{B_\rho(x_0) \times (\bR^n \setminus B_\rho(x_0)) \cap \{u(x,t) \,\geq \,u(y,t)\}}\frac{u(x,t)^{p-1-\varepsilon}}{|x-y|^{n+sp}}\varphi^p(x,t)\dxyt \\[4mm]
&\leq \Lambda \left(\sup \limits_{x\in \supp\,\varphi(\cdot,t)}\int_{\bR^n \setminus B_\rho(x_0)}\frac{\dy}{|x-y|^{n+sp}}\right)\iint_{Q_\rho^-(z_0)}u(x,t)^\alpha\varphi^p(x,t)\dxt.
\end{align*}
Joining the preceding estimates for $\mathbf{V}_1$ and $\mathbf{V}_2$ we get
{\small
\begin{align*}
\mathbf{V} &\leq 
-\Lambda c(p) \zeta(\varepsilon)\int_{t_0-\rho^p}^{t_0}\iint_{B_\rho(x_0) \times B_\rho(x_0)}\frac{\Big|\varphi(x,t)u(x,t)^{\frac{\alpha}{p}}-\varphi(y,t)u(y,t)^{\frac{\alpha}{p}}\Big|^p}{|x-y|^{n+sp}} \dxyt \\[4mm]
&\quad +\Lambda\left(\zeta(\varepsilon)+1+\varepsilon^{-(p-1)}\right)\int_{t_0-\rho^p}^{t_0}\iint_{B_\rho(x_0) \times B_\rho(x_0)}\frac{\big|\varphi(x,t)-\varphi(y,t)\big|^p\left(u(x,t)^\alpha+u(y,t)^\alpha\right)}{|x-y|^{n+sp}}\dxyt \\[4mm]
&\quad \quad +2\Lambda \left(\sup \limits_{x\in \supp\,\varphi(\cdot,t)}\int_{\bR^n \setminus B_\rho(x_0)}\frac{\dy}{|x-y|^{n+sp}}\right)\iint_{Q_\rho^-(z_0)}u(x,t)^\alpha\varphi^p(x,t)\dxt.
\end{align*}
}

\noindent
Inserting this estimate back to~\eqref{a9}, we finally arrive at the desired estimate and therefore, the lemma is completely proved.
\end{proof}
%%%%%%%%%%%%%%%%%%%%%%%%%%%%%%%%%%%%%%%%
%%%%%%%%%%%%%%%%%%%%%%%%%%%%%%%%%%%%%%%%
%%%%%%%%%%%%%%%%%%%%%%%%%%%%%%%%%%%%%%%%
%%%%%%%%%%%%%%%%%%%%%%%%%%%%%%%%%%%%%%%%

\section{Proof of Lemmata~\ref{log-type Caccioppoli lemma} and~\ref{supsub lemma}}\label{Appendix B}

In this final appendix, we report the proof of Lemmata~\ref{log-type Caccioppoli lemma} and~\ref{supsub lemma}.
\begin{proof}[\normalfont\textbf{Proof of Lemma~\ref{log-type Caccioppoli lemma}}]
Since by $u \geq m >0$ in $\bR^n \times (t_1,t_2)$ there holds that, for all $t \in (t_1,t_2)$
\begin{align}~\label{b1}
[u^{p-1}]_h= \frac{1}{h}\int_{0}^{t}e^{\frac{s-t}{h}}u(s)^{p-1}\ds &\geq m^{p-1}\left(\frac{1}{h}\int_0^te^{\frac{s-t}{h}}\ds\right) \notag \\[3mm]
&=m^{p-1}(1-e^{-\frac{t}{h}}).
\end{align}
This together with an elementary estimate $|\log s| \leq \dfrac{|s-1|}{\min\{s,1\}}$ for $s>0$ implies that 
\begin{align*}
\Big|\log[u^{p-1}]_h-\log u^{p-1}\Big|=\left|\log\left(\frac{[u^{p-1}]_h}{u^{p-1}}\right)\right| &\leq \frac{\big|[u^{p-1}]_h-u^{p-1}\big|}{\min\{[u^{p-1}]_h,\,u^{p-1}\}} \\[3mm]
&\leq \frac{1}{m^{2(p-1)}(1-e^{-\frac{t}{h}}) }\big|[u^{p-1}]_h-u^{p-1}\big|
\end{align*}
for every $(x,t) \in \Omega_{t_1,t_2}$. Therefore, Lemma~\ref{mollification lemma}-(ii) concludes that
\[
\Big\|\log[u^{p-1}]_h-\log u^{p-1}\Big\|_{L^{\frac{p}{p-1}}(\Omega_{t_1,t_2})} \leq \frac{1}{m^{2(p-1)}(1-e^{-\frac{t_1}{h}})}\Big\|[u^{p-1}]_h-u^{p-1}\Big\|_{L^{\frac{p}{p-1}}(\Omega_{t_1,t_2})} \to 0,
\]
as desired.
\end{proof}
\begin{proof}[\normalfont \textbf{Proof of Lemma~\ref{supsub lemma}}]
Again, by~\eqref{b1}
\[
\big|[u^{p-1}]_h^{-1}-u^{1-p}\big| \leq \frac{1}{m^{2(p-1)}(1-e^{-\frac{t_1}{h}})}\big|[u^{p-1}]_h-u^{p-1}\big|
\]
holds whenever $(x,t) \in \Omega_{t_1,T}$. Hence
\begin{align*}
\Big\|[u^{p-1}]_h^{-1}-u^{1-p}\Big\|_{L^{\frac{p}{p-1}}(\Omega_{t_1,T})} & \leq \frac{1}{m^{2(p-1)}(1-e^{-\frac{t_1}{h}})}\Big\|[u^{p-1}]_h-u^{p-1}\Big\|_{L^{\frac{p}{p-1}}(\Omega_{t_1,T})} \to 0,\end{align*}
finishing the proof.
\end{proof}

%%%%%%%%%%%%%REFFERENCES%%%%%%%%%%%%%%%%%%%%%%%%%%%%%
%%%%%%%%%%%%%%%%%%%%%%%%%%%%%%%%%%%%
%%%%%%%%%%%%%%%%%%%%%%%%%%%%%%%%%%%%

\end{document}